\documentclass[11pt]{amsart}
\usepackage{amssymb}
\usepackage{latexsym}
\usepackage{amsmath}
\usepackage{amsfonts}
\usepackage[colorlinks]{hyperref}
\usepackage{enumitem}
\usepackage{fancyhdr}
\usepackage{cleveref}
\usepackage{mathrsfs}

\usepackage{color}
\theoremstyle{plain}
\newtheorem{theorem}{Theorem}[section]
\newtheorem{untheorem}{Theorem}

\newtheorem{lemma}[theorem]{Lemma}
\newtheorem{proposition}[theorem]{Proposition}

\newtheoremstyle{remark}
    {} 
    {} 
    {}          
    {}          
    {\bfseries} 
    {.}         
    {.5em}      
    {}          

\theoremstyle{remark}
\newtheorem{remark}{\emph{\textbf{Remark}}}[section]

\newtheoremstyle{example}
    {\dimexpr\topsep/2\relax} 
    {\dimexpr\topsep/2\relax} 
    {}          
    {}          
    {\bfseries} 
    {.}         
    {.5em}      
    {}          

\theoremstyle{example}
\newtheorem{example}{\emph{\textbf{Example}}}[section]

\newtheoremstyle{definition}
    {\dimexpr\topsep/2\relax} 
    {\dimexpr\topsep/2\relax} 
    {}          
    {}          
    {\bfseries} 
    {.}         
    {.5em}      
    {}          

\theoremstyle{definition}

\newtheorem*{similartheorem*}{Theorem \dualnumber{$'$}}

\numberwithin{equation}{section}

\allowdisplaybreaks

\begin{document}
\title{A sharp form of the Marcinkiewicz Interpolation Theorem for Orlicz spaces}
{\let\thefootnote\relax}

%
\author{Ron Kerman}
\author{Rama Rawat and Rajesh K. Singh}
%

\address{Ron Kerman: Department of Mathematics, Brock University, St. Catharines, Ontario, L2S 3A1, Canada}
\email{rkerman@brocku.ca}
\address{Rama Rawat: Department of Mathematics and Statistics, Indian Institute of Technology, Kanpur-208016}
\email{rrawat@iitk.ac.in}
\address{Rajesh K. Singh: Department of Mathematics and Statistics, Indian Institute of Technology, Kanpur-208016}
\email{rajeshks@iitk.ac.in}

\pagestyle{headings}

\begin{abstract}
An extension of  Marcinkiewicz Interpolation Theorem, allowing intermediate spaces of Orlicz type, is proved. This generalization  yields a necessary and sufficient condition so that every quasilinear operator, which maps the set, $S(X,\mu)$, of all $\mu$-measurable simple functions on $\sigma$- finite measure space $(X,\mu)$ into $M(Y,\nu)$, the class of $\nu$-measurable functions on $\sigma$- finite measure space $(Y,\nu)$, and satisfies endpoint estimates of type:
 $1 < p< \infty$, $1 \leq r < \infty$,
\begin{equation*}
\lambda \, \nu \left( \left\lbrace y \in Y : |(Tf)(y)| > \lambda   \right\rbrace \right)^{\frac{1}{p}} \leq C_{p,r} \left( \int_{\mathbb{R_+}} \mu \left( \left\lbrace x \in X : |(f)(x)| > t   \right\rbrace \right)^{\frac{r}{p}} t^{r-1}dt \right)^{\frac{1}{r}},
\end{equation*}
for all $f \in S(X,\mu)$ and $\lambda \in \mathbb{R_+}$; is bounded from an Orlicz space into another.

%
%
\end{abstract}

\maketitle

\section{Introduction}\label{Introduction}

Let $(X,\mu)$ and $(Y,\nu)$ be two $\sigma$-finite measure spaces  and denote by $T$ a quasilinear operator  that maps the set, $S(X,\mu)$, of all $\mu$-measurable simple functions on $X$ into $M(Y,\nu)$, the class of $\nu$-measurable functions on $Y$. 
 
An important special case of the classical Marcinkiewicz interpolation theorem asserts that, if $0<p_0, p_1 < \infty$  with $p_0 < p_1$, then every quasilinear operator $T$ of weak-types $(p_0,p_0)$ and $(p_1,p_1)$, namely, satisfying the inequalities
\begin{align}\label{weak type (p,p)}
\lambda \, \nu \left( \{ y \in Y : |(Tf)(y)|> \lambda \} \right)^{\frac{1}{p_i}} \leq M_i \left( \int_{X} |f(x)|^{p_i} d\mu(x) \right)^{\frac{1}{p_i}}, \ \ \ i=0,1,
\end{align}
in which the positive constants $M_0$ and $M_1$ are independent of $f \in S(X,\mu)$  and $\lambda > 0$, is bounded on Lebesgue space $L_{p_{\theta}}(X,\mu)$, provided
\begin{equation*}
\frac{1}{p_{\theta}}= \frac{1-\theta}{p_0} + \frac{\theta}{p_1},
\end{equation*}
for some $\theta \in (0,1)$ {\cite{Ma39},\cite{Zy56}}.

In his $1956$ paper Zygmund proved the following partial generalization of Marcinkiewicz interpolation theorem, the formulation of which, \cite{Zy57}, he also attributes to Marcinkiewicz. The principal result of this paper, Theorem A, is partly modelled on this generalization.
\begin{theorem}\label{Marcinkiewicz Zygmumnd theorem part1}
Let $(X,\mu)$ and $(Y,\nu)$ be \textbf{finite} measure spaces and suppose $T$ is a quasilinear operator  of weak-types $(p_{0},p_{0})$ and $(p_{1},p_{1})$, $1 \leq p_{0} < p_{1} < \infty$. Let  $\Phi$ be a  increasing continuous function on $\mathbb{R_+}=(0,\infty)$ satisfying $\Phi(0^{+})=0$. Then, $Tf$ is defined for every $f$ with $\int_{X}\Phi(|f(x)|)d\mu(x) < \infty$ and
\begin{equation}\label{modular boundedness by Zygmund part1}
\int_{Y} \Phi(|(Tf)(y)|) d\nu(y) \leq K\int_{X} \Phi(|f(x)|)d\mu(x)  + K,
\end{equation}
$K>0$ being independent of $f$, provided
\begin{equation}\label{Marcinkiewicz Stromberg Zygmund condditions}
\begin{aligned}
\Phi(2t)= O (\Phi(t)),\\
\int_{1}^t \frac{\Phi(s)}{s^{p_{0}+1}} ds =  O \left( \frac{\Phi(t)}{t^{p_0}} \right)\\
\end{aligned}
\end{equation}
and
\begin{equation}
\int_t^{\infty} \frac{\Phi(s)}{s^{p_{1}+1}} ds = O \left( \frac{\Phi(t)}{t^{p_1}} \right),
\end{equation}
as $t \rightarrow \infty$.
\end{theorem}

Str\"omberg \cite{Str79} proved a similar result when  both the measure spaces $(X,\mu)$ and $(Y,\nu)$ are $(\mathbb{R}^n,m)$, $m$ being Lebesgue measure on $\mathbb{R}^n$. His result can be adapted to incorporate \emph{two}  increasing continuous functions $\Phi_1$ and $\Phi_2$, and then it reads

\begin{theorem}[Str\"omberg, {\cite{Str79}}]\label{Stromberg's general result}
Let $(X,\mu)$ and $(Y,\nu)$ be totally $\sigma$-finite measure spaces with $\mu(X)=\nu(Y)= \infty$. Suppose $T$ is a  quasilinear operator from $S(X,\mu)$ into $M(Y,\nu)$ that is of weak-types $(p_{0},p_{0})$ and $(p_{1},p_{1})$, $1 \leq p_{0} < p_{1} < \infty$. Then, 
\begin{equation}
\int_{Y} \Phi_1(|(Tf)(y)|) d\nu(y) \leq \int_{X} \Phi_2(|f(x)|)d\mu(x),
\end{equation}
where $K>0$ is independent of $f \in S(X,\mu)$, provided there exist $A,B>0$ such that
\begin{equation}\label{Zygmund and Stromberg conditions labeled first time}
\begin{aligned}
t^{p_0}\int_{0}^t \frac{\Phi_1(s)}{s^{p_0+1}} ds \leq \Phi_2(A t)\\
t^q\int_t^{\infty} \frac{\Phi_1(s)}{s^{q+1}} ds \leq  \Phi_2(B t),
\end{aligned}
\end{equation}
for all $t \in \mathbb{R_+}$. 
\end{theorem}
The conditions in (\ref{Zygmund and Stromberg conditions labeled first time}) will be referred to as \emph{Zygmund-Str\"omberg conditions} in the sequel.

A (complete) generalization of Marcinkiewicz interpolation theorem in a different direction was given by Stein and Weiss \cite{StW59}. They arrived at the same conclusion as Marcinkiewicz assuming that the weak-type inequalities hold only for characteristic functions, $\chi_{E}$, of sets $E \subset X$ with $\mu(E)< \infty$. This leads to the so-called restricted weak-type conditions. Calder\'on in \cite{Ca66} (see also Hunt \cite{Hu64}) showed that if a nonnegative sublinear operator $T$ is of restricted weak-type $(p,q)$ then it satisfies the stronger inequality
\begin{equation}\label{restricted weak type (p,q)}
\lambda \, \nu \left( \left\lbrace y \in Y : |(Tf)(y)| > \lambda   \right\rbrace \right)^{\frac{1}{q}} \leq M \int_{\mathbb{R_+}}  \mu \left( \left\lbrace x \in X : |f(x)| > t   \right\rbrace \right)^{\frac{1}{p}}dt.
\end{equation}

In \cite{Ci98} and \cite{Ci99}, Cianchi obtained an interpolation theorem, in the spirit of those of Zygmund and Str\"omberg, for quasilinear operators of \emph{restricted} weak-types $(p_0,p_0)$ and $(p_1,p_1)$, $1 \leq p_0 < p_1 < \infty$. It concerns two increasing functions on $\mathbb{R_+}$, $\Phi_1$ and $\Phi_2$, that are Young functions. A Young function, $\Phi$, is a function from $\mathbb{R_+}$ into $\mathbb{R_+}$ having the form
\begin{equation}
\Phi(t)= \int_0^t \phi(s)ds,
\end{equation}
for all  $t \in \mathbb{R_+}$, in which $\phi : \mathbb{R_+} \rightarrow \mathbb{R_+}$ is an increasing, left-continuous, with $\phi(0^{+})=0$ and which is neither identically zero nor identically infinite.


Given a Young function $\Phi$ and a totally $\sigma$-finite measure space $(X,\mu)$ one defines the Orlicz class
\begin{equation*}
L_{\Phi}(X,\mu)= \left\lbrace f \in M(X,\mu): \int_{X} \Phi \left( k|f(x)|\right) d\mu(x) < \infty \, \, \, \text{for some } \, k \in \mathbb{R_+}  \right\rbrace.
\end{equation*}
Under the gauge norm
\begin{equation}\label{gauge norm}
\| f \|_{L_{\Phi}(X,\mu)} = \inf \left\lbrace  \lambda >0 : \int_{X} \Phi \left( \frac{|f(x)|}{\lambda} \right) d \mu(x) \leq 1   \right\rbrace,
\end{equation}
$L_{\Phi}(X,\mu)$ becomes a Banach function space as in \cite[Theorem $8.9$, p. 269]{BS88}.

We observe here that if the $\Phi_1$ and $\Phi_2$ in \Cref{Stromberg's general result} are Young functions, then the assumptions of the theorem guarantee the norm inequality which is the subject of

\begin{theorem}[Cianchi, {\cite{Ci98}, \cite{Ci99}}]
Let $(X,\mu)$ and $(Y,\nu)$ be nonatomic $\sigma$-finite measure spaces with (for simplicity) $\mu(X)=\nu(Y)= \infty$. Fix the indices $p_0$ and $p_1$, $1 \leq p_0 < p_1 < \infty$. Suppose $\Phi_{i}(t)= \int_0^t \phi_{i}(s)ds$ are Young functions with complementary functions $\Psi_{i}(t)= \int_0^t \phi_{i}^{-1}(s)ds$, $i=1,2$. Assume $T$ is a  quasilinear operator from $S(X,\mu)$ into $M(Y,\nu)$ which is of restricted weak-types $(p_{0},p_{0})$ and $(p_{1},p_{1})$, that is (\ref{restricted weak type (p,q)}) holds for pairs $(p_0,p_0)$ and $(p_1,p_1)$. Then, there exists a constant $C>0$, independent of $f \in L_{\Phi_2}(X,\mu)$, with
\begin{equation*}
\| Tf \|_{L_{\Phi_1}(Y,\nu)} \leq C \| f \|_{L_{\Phi_2}(X,\mu)},
\end{equation*}
if, for some $K>0$,
\begin{align}
\left( \int_{Kt}^{\infty} \left( \frac{s}{\Psi_1(s)} \right)^{p_1-1} ds \right)^{\frac{1}{p_1}} \left( \int_{0}^{t} \frac{\Psi_2(s)}{s^{p_1'+1}} ds \right)^{\frac{1}{p_1'}} \leq K, \hspace*{3.5cm} \nonumber\\
 and \hspace*{15cm} \\
t \int_0^t \frac{\Phi_1(s)}{s^{2}}ds \leq \Phi_2(Kt),  \hspace*{5.8cm} \nonumber
\end{align}
or
\begin{equation}\label{Cianchi's condition for Hpr}
\left( \int_{Kt}^{\infty} \left( \frac{s}{\Phi_2(s)} \right)^{p_0'-1} ds \right)^{\frac{1}{p_0'}} \left( \int_{0}^{t} \frac{\Phi_1(s)}{s^{p_0+1}} ds \right)^{\frac{1}{p_0}} \leq K,
\end{equation}
for all $t \in \mathbb{R_+}$, depending on whether $p_0=1$ or $1< p_0 < \infty$.
\end{theorem}

In this paper, we have considered operators such as
\begin{equation*}
f(x) \rightarrow \int_t^{\infty} f^{*_{\mu}}(s) \frac{ds}{s},
\end{equation*}
 where $f^{*_{\mu}}$ is the decreasing rearrangement of $f$ (see (\ref{decreasing rearrangements S}). These are not quasilinear, however, they are $r$-quasilinear, in the sense that
\begin{equation}\label{nu quasilinear}
 \left[ T(f_1+f_2) \right]^{*_{\nu}}(t) \leq C \left[ (Tf_1)^{*_{\nu}}(ct) + (Tf_2)^{*_{\nu}}(ct) \right],
\end{equation}
in which $C>0$, $0<c<1$ are independent of $f_1, f_2 \in S(X,\mu)$ and $t \in \mathbb{R_+}$. We observe that a quasilinear operator is $r$-quasilinear.

We now prepare to state our principal result, \Cref{Combined main theorem}, below. In it we consider weak-type conditions on a so-called $r$-quasilinear operator (see (\ref{nu quasilinear})) that are refinements of the weak type $(p,p)$ and restricted weak-type $(p,p)$ conditions. Thus, for $\sigma$- finite measure spaces $(X,\mu)$ and $(Y,\nu)$ and an $r$-quasilinear operator $T$ from $S(X,\mu)$ into $M(Y,\nu)$ and $1 \leq p< \infty$, $1 \leq r < \infty$, we are interested in the following weak type conditions:
\begin{equation}\label{type (p,r)}
\lambda \, \nu \left( \left\lbrace y \in Y : |(Tf)(y)| > \lambda   \right\rbrace \right)^{\frac{1}{p}} \leq C_{p,r} \left( \int_{\mathbb{R_+}} \mu \left( \left\lbrace x \in X : |(f)(x)| > t   \right\rbrace \right)^{\frac{r}{p}} t^{r-1}dt \right)^{\frac{1}{r}},
\end{equation}
where $C_{p,r}=C_{p,r}(T)>0$ is independent of $f \in S(X,\mu)$ and $\lambda \in \mathbb{R_+}$. For $r=1$ and $r=p$ this is (\ref{restricted weak type (p,q)}) and (\ref{weak type (p,p)}), respectively. When $1<r<p$, it will be seen that (\ref{type (p,r)}) is intermediate in strength between the two; when $r>p$, then (\ref{type (p,r)}) is stronger than (\ref{weak type (p,p)}).


Finally, when $1 \leq p_0 < p_1< \infty$ and $1 \leq r_0,r_1 < \infty$, we denote by $W((p_0,r_0),(p_1,r_1);\mu,\nu)$ the class of all $r$-quasilinear operators $T$, mapping $S(X,\mu)$ into $M(Y,\nu)$, which satisfy weak type estimates (\ref{type (p,r)}) for the pairs $(p_0,r_0)$ and $(p_1,r_1)$. Also, denote by $W((p_0,r_0),(\infty, \infty);\mu,\nu)$ the class of all $r$-quasilinear operators $T$, mapping $S(X,\mu)$ into $M(Y,\nu)$, which satisfy the inequality (\ref{type (p,r)}) 
for the pair $(p_0,r_0)$ and are bounded from $L_{\infty}(X,\mu)$ to $L_{\infty}(Y,\nu)$.

In this paper, we assume that the measure spaces $(X,\mu)$ and $(Y,\nu)$ are such that $\mu(X)=\nu(Y)=\infty$.

\begin{untheorem}\label{Combined main theorem}
Let $(X,\mu)$ and $(Y,\nu)$ be   $\sigma$-finite measure spaces with $\mu(X)=\nu(Y)=\infty$, the latter being nonatomic and separable. Fix the indices $p_0$, $p_1$, $r_0$ and $r_1$, with $1 < p_0 < p_1 < \infty$ and $1 \leq r_0, r_1 < \infty$. Suppose $\Phi_{i}(t)= \int_0^t \phi_{i}(s)ds$, $i=0,1$, are Young functions.
%
Then, the following are equivalent:
\begin{enumerate}
\item[(1)] To each $T \in W((p_0,r_0),(p_1,r_1);\mu,\nu)$ there corresponds $C>0$, depending on $C_{p_0,r_0}$ and $C_{p_1,r_1}$, such that
\begin{equation*}
\| Tf \|_{L_{\Phi_1}(Y,\nu)} \leq C \| f \|_{L_{\Phi_2}(X,\mu)},
\end{equation*}
for all $f \in S(X,\mu)$;
\item [(2)] To each $T \in W((p_0,r_0),(p_1,r_1);\mu,\nu)$ there corresponds $K>0$, depending on $C_{p_0,r_0}$ and $C_{p_1,r_1}$ such that
\begin{equation*}
\int_{Y} \Phi_1(|(Tf)(y)|) d\nu(y) \leq \int_{X} \Phi_2(K|f(x)|)d\mu(x),
\end{equation*}
for all $f \in S(X,\mu)$;
\item [(3)] There exist constants $L,B,D>0$ such that, for all $t>0$,
\begin{align*}
\chi_{[0,\infty)} (r_0-p_0) \, &  \frac{t^{p_0}}{\Phi_2(t)}  \int_0^t \frac{\Phi_1(s)}{s^{p_0+1}} ds \, \, \\
 & + \, \,  \chi_{(0,\infty)}(p_0-r_0) \left(\int_{t}^{\infty}  \frac{\phi_2(Ls)}{\Phi_2(Ls)^{\rho_0'}} s^{r \rho_0'} ds \right)^{\frac{1}{\rho_0'}} \left( \int_{0}^{t}\frac{ \Phi_1(s) }{s^{p_0+1}}ds \right)^{\frac{1}{\rho_0}} \leq B
\end{align*}
and
\begin{align*}
\chi_{[0,\infty)}  (r_1-p_1) \, & \frac{t^{p_1}}{\Phi_2(t)}  \int_0^t \frac{\Phi_1(s)}{s^{p_1+1}} ds \, \, \\
 & + \, \,  \chi_{(0,\infty)}(p_1-r_1) \left(\int_{0}^{t}  \frac{\phi_2(Ls)}{\Phi_2(Ls)^{\rho_1'}} s^{r\rho_1'} ds \right)^{\frac{1}{\rho_1'}} \left( \int_{t}^{\infty}\frac{ \Phi_1(s) }{s^{p_1+1}}ds \right)^{\frac{1}{\rho_1}} \leq D,
\end{align*}
\end{enumerate}
where $\rho_0= p_0/r_0$ and $\rho_1= p_1/r_1$.
\end{untheorem}

We remark here that for $r_0= \infty$, (\ref{type (p,r)}) should be replaced by
\begin{equation*}
\lambda \, \nu \left( \left\lbrace y \in Y : |(Tf)(y)| > \lambda   \right\rbrace \right)^{\frac{1}{p}} \leq C_{p,\infty} \, \sup_{t>0} \, t \, \nu \left( \left\lbrace x \in X : |f(x)| > t   \right\rbrace \right)^{\frac{1}{p}}.
\end{equation*}
An $r$-quasilinear operator $T$ satisfying such an inequality is said to be of weak-type $(p,\infty)$. The principal result in the paper \cite{KPP14} of Kerman, Phipps and Pick yields an analogue of \cref{Combined main theorem} for the class $ W((p_0,\infty),(p_1,\infty);\mu,\nu)$, $1<p_0<p_1<\infty$, and interestingly Zygmund-Str\"omberg condition (\ref{Zygmund and Stromberg conditions labeled first time}) characterizes this class as well.


Let $(X,\mu)$, $(Y,\nu)$, $p_0$, $r_0$, $p_1$ and $r_1$ be as in \Cref{Combined main theorem}. To study the  behaviour of an $r$-quasilinear operator $T$ in $W((p_0,r_0),(p_1,r_1);\mu,\nu)$ we transfer  considerations to the measure space $(\mathbb{R_+},m)$, $m$ being Lebesgue measure. This is done using rearrangements.

So, if $f \in M(X,\mu)$, say, its distribution function, $\lambda_{f,\mu}$, with respect to $\mu$, is given at $s \in \mathbb{R_+}$ by  $\lambda_{f,\mu}(s)= \mu  \left( \left\lbrace  x \in X : |f(x)| > s  \right\rbrace \right)$.
This nonnegative function is nonincreasing on $\mathbb{R_+}$ and so has a unique right-continuous generalized inverse,
\begin{equation}\label{decreasing rearrangements S}
f^{*_{\mu}} = \lambda_{f,\mu}^{-1},
\end{equation}
called the decreasing rearrangement of $f$.  We shall also use notation  $f^{*}$ instead of $f^{*_{\mu}}$ whenever there is no possible confusion.

Following Calder\'on \cite{Ca66} we show that given $T$ in $W((p_0,r_0),(p_1,r_1);\mu,\nu)$ and $f \in S(X,\mu)$,

\begin{equation}\label{eq:Dominance of Calderon operator for joined class introduction S}
\begin{split}
(Tf)^*(t) & \leq K  \left[    \left( t^{-\frac{r_0}{p_0}} \int_0^{t}  f^*(s)^{r_0} s^{\frac{r_0}{p_0}-1}ds \right)^{\frac{1}{r_0}}  + \left( t^{-\frac{r_1}{p_1}} \int_{t}^{\infty}  f^*(s)^{r_1} s^{\frac{r_1}{p_1}-1}ds \right)^{\frac{1}{r_1}}    \right]\\
& := K \left[ (H^{p_0,r_0}f^*)(t)+ (H_{p_1,r_1}f^*)(t) \right],
\end{split}
\end{equation}
for all $t \in \mathbb{R_+}$, where the rearrangement of $Tf$ is with respect to $\nu$, the rearrangement of $f$ is with respect to $\mu$
 and $K>0$ depends on $T$, but not on $f$. The operators $H^{p_0,r_0}$ and $H_{p_1,r_1}$ are, respectively, natural generalizations of the classical Hardy operator and of its dual operator. 
 
From the inequality (\ref{eq:Dominance of Calderon operator for joined class introduction S}) and other considerations we are able to show that, for a pair of Young functions, $\Phi_1$ and $\Phi_2$, one has 
\begin{equation}\label{norm inequality for generic operator S}
\| Tf \|_{L_{\Phi_1}(Y,\nu)} \leq C(T) \| f \|_{L_{\Phi_2}(X,\mu)}, \ \ \ f \in S(X,\mu),
\end{equation}
for \emph{all} $T \in W((p_0,r_0),(p_1,r_1);\mu,\nu)$ if and only if

\begin{align}\label{equivalent two norm inequalities Hpr and Hqr S}
\parallel H^{p_0,r_0}g \parallel _{L_{\Phi_1}(\mathbb{R_+},m)} \leq C_1 \parallel g \parallel_{L_{\Phi_2}(\mathbb{R_+},m)} \hspace*{5cm} \nonumber\\
 and \hspace*{15cm} \\
\parallel H_{p_1,r_1}g \parallel _{L_{\Phi_1}(\mathbb{R_+},m)} \leq C_2 \parallel g \parallel_{L_{\Phi_2}(\mathbb{R_+},m)}, \hspace*{4.9cm} \nonumber
\end{align}
for all nonnegative, nonincreasing $g$ on $\mathbb{R_+}$.

The inequalities in (\ref{equivalent two norm inequalities Hpr and Hqr S}) suggest working with the class $W((p_0,r_0),(\infty,\infty);\mu,\nu)$ and $W((1,1),$ $(p_1,r_1);\mu,\nu)$ instead of $W((p_0,r_0),(p_1,r_1);\mu,\nu)$. This, indeed, allows us to solve the problem at hand, Theorem A. 

In \Cref{Class W(prInfinity)} we take up the class $W((p_0,r_0),(\infty,\infty);\mu,\nu)$ and the first inequality in (\ref{equivalent two norm inequalities Hpr and Hqr S}).

This in turn leads us to considering the boundedness of the operator $H^{p_0,r_0}$ between Orlicz spaces. As the operator $H^{p_0,r_0}$ is dilation-commuting, the first inequality in (\ref{equivalent two norm inequalities Hpr and Hqr S}) is equivalent to the modular inequality

%
%
\begin{equation}\label{modular inequality for Hpr a key observation S}
  \int_{\mathbb{R_+}}\Phi_1((H^{p_0,r_0}f^*)(t))dt \leq \int_{\mathbb{R_+}}\Phi_2(Kf^*(s))ds ,
\end{equation}
where $f \in M(X,\mu)$.

Further work allows us to reduce this inequality to a weighted dual Hardy inequality for nonnegative, nonincreasing functions $g$, namely, to an inequality of the form
\begin{equation*}
\int_{\mathbb{R_+}}     \left(   \int_{x}^{\infty} g(y)dy \right)^{p_0/r_0}  w(x) dx 
\leq B  \int_{\mathbb{R_+}}  g(y) ^{p_0/r_0}    v(y)dy,
\end{equation*}
involving certain weights $w$ and $v$ on $\mathbb{R_+}$.

As seen from the statement of \Cref{Combined main theorem} we are then able to show that (\ref{modular inequality for Hpr a key observation S})  (and hence the first  inequality in (\ref{equivalent two norm inequalities Hpr and Hqr S})) holds if and only if
\begin{equation*}
\left(\int_{t}^{\infty}  \frac{\phi_2(Ls)}{\Phi_2(Ls)^{\rho_0'}} s^{r \rho_0'} ds \right)^{\frac{1}{\rho_0'}} \left( \int_{0}^{t}\frac{ \Phi_1(s) }{s^{p_0+1}}ds \right)^{\frac{1}{\rho_0}} \leq A,
\end{equation*}
for all $t \in \mathbb{R_+}$, where $\rho_0= \frac{p_0}{r_0}$, $1 \leq r_0 <p_0$.

By other means we prove (\ref{modular inequality for Hpr a key observation S}) equivalent to
\begin{equation*}
t^{p_0} \int_0^t \frac{\Phi_1(s)}{s^{p_0+1}} ds \leq B \Phi_2(Kt),
\end{equation*}
for all $t \in \mathbb{R_+}$, when $r_0 \geq p_0$.

In \Cref{Class W(11qr)}, we consider the class $W((1,1),(p_1,r_1);\mu,\nu)$. The corresponding Calder\'on operator that arises in this case is $P+H_{p_1,r_1}$, where operator $P=H^{1,1}$.
The necessary and sufficient conditions for the boundedness of $P$ between Orlicz spaces are already known, for example in \cite{Ci99}. The treatment of the operator $H_{p_1,r_1}$ is is similar to the one for $H^{p_0,r_0}$. The final result, then, combines the conditions for $P$ and $H_{p_1,r_1}$ to arrive at conditions for the class $W((1,1),(p_1,r_1);\mu,\nu)$.


In Section 5, we present the Marcinkiewicz interpolation theorem  for the joint class $W((p_0,r_0),$ $(p_1,r_1);\mu,\nu)$, $1 < p_0 < p_1 < \infty$, $1 \leq r_1, r_2 < \infty$, combining the results of \Cref{Class W(prInfinity)} and \Cref{Class W(11qr)}.

 In the concluding Section 6, we give an example to  compare various conditions obtained in \Cref{Combined main theorem}. 


\section{Background}\label{Background}
\subsection{}
Suppose $(X,\mu)$ is a $\sigma$-finite measure space. Given $f \in M(X,\mu)$, we define the decreasing rearrangement, $f^{*_{\mu}}$, of $f$ by
\begin{equation*}
f^{*_{\mu}}(t)=\inf \{s>0:\lambda_{f,\mu}(s)\leq t\},
\end{equation*}
where, $t \in \mathbb{R_+}$ and $\lambda_{f,\mu}$ is the \emph{distribution function} of $f$ given by
\begin{equation}\label{distr fn}
\lambda_{f,\mu}(s)=\mu \left( \{x \in X :|f(x)| > s \} \right), \ \ \ \text{for} \ s \in \mathbb{R_+}.
\end{equation}
We remark here that the dependence of $f^*$ on $\mu$ will usually be clear from the context in which it appears. When we wish to emphasize the dependence we will use the notation $f^{*_{\mu}}$ rather than $f^*$.

Two functions $f \in M(X,\mu)$ and $g \in M(Y,\nu)$ are said to be \emph{equimeasurable} if they have the same distribution function, that is, if $\lambda_{f,\mu}(s)= \lambda_{g,\nu}(s)$ for all $s \in \mathbb{R_+}$.

The decreasing rearrangement $f^{*}$ satisfies the following inequality of Hardy and Littlewood: For $f,g \in M_{+}(X,\mu)$,
\begin{equation}
\int_{X}f(x)g(x)d\mu(x) \leq \int_{\mathbb{R_+}}f^*(t)g^*(t)dt.
\end{equation}

The operation of rearrangement is not sublinear though it satisfies
\begin{equation}\label{eq:subdr}
(f+g)^*(t_1 + t_2) \leq f^*(t_1) + g^*(t_2),
\end{equation}
where $f,g \in M(X,\mu)$ and $t_1,t_2 \in  \mathbb{R_+}.$

However, the operator $f \mapsto f^{**}$ is sublinear, namely,
\begin{equation}
(f+g)^{**}(t) \leq f^{**}(t) + g^{**}(t),
\end{equation}
for all $t \in \mathbb{R_+}$, where the maximal function $f^{**}$ of $f$ (or the Hardy averaging operator at $f$ ) is defined, at $t \in \mathbb{R_+}$, by
\begin{equation}
f^{**}(t)= \frac{1}{t}\int_0^t f^*(s)ds.
\end{equation}
\subsection{}
Let $(X,\mu)$ be a $\sigma-$finite measure space and suppose $0< p< \infty $, $0< r \leq \infty $. The \emph{Lorentz space} $L_{p,r}(X,\mu)$ consists of all $f$ in $M(X,\mu)$ for which the quantity
\begin{equation*}
 \| f \|_{p,r} = \begin{cases}
              {\displaystyle \left( {\frac{r}{p}} \int_{\mathbb{R_+}} \left( t^{\frac{1}{p}}f^*(t) \right) ^r \frac{dt}{t} \right)^{\frac{1}{r}} },  & 0<r< \infty, \\
              {\displaystyle~   \sup_{0<t< \infty} t^{\frac{1}{p}}f^*(t) },              & r = \infty,
       \end{cases}
\end{equation*}
is finite.

A useful alternative expression for $\| f \|_{p,r}$, $r< \infty$, is
\begin{equation*}
\| f \|_{p,r} = \left( r \int_{\mathbb{R_+}} \lambda_{f,\mu}(s)^{\frac{r}{p}} s^{r-1} ds \right)^{\frac{1}{r}}.
\end{equation*}
This indeed was the form of $\| \cdot \|_{p,r}$ used in Introduction.

For any fixed $p$, $0<p<\infty$, the Lorentz space $L_{p,r}$ gets bigger as the secondary exponent $r$ increases: If $0<r_1<r_2 \leq \infty$. Then,
\begin{equation}\label{Lorentz mono}
\|f\|_{p,r_2} \leq \|f\|_{p,r_1},
\end{equation}
for all $f \in M(X,\mu)$. In particular, one has the imbedding $L_{p,r_1}(X,\mu) \hookrightarrow L_{p,r_2}(X,\mu) $.

\subsection{}
A Young function is convex and that for $t \in \mathbb{R_+}$,
\begin{equation}\label{eq:ygp}
\Phi(t) \leq t\phi(t) \leq \Phi(2t).
\end{equation}  
We associate to the Orlicz space $L_{\Phi}$ an another Orlicz space $L_{\Psi}$ that has the same relationship to $L_{\Phi}$ as the Lebesgue space $L_{p'}$ does to the Lebesgue spaces $L_{p}$, where $p'= \textstyle\frac{p}{p-1}$.

Let $\Phi(t)=\int_0^t \phi(s)ds$ be a Young function. Let $\psi$ be \emph{left-continuous inverse} of $\phi$, that is, for $t \in \mathbb{R_+}$
\begin{equation}
\psi(t)= \inf \{ s:\phi(s) \geq t \}.
\end{equation}
 Then the function $\Psi$ defined as 
\begin{equation*}
\Psi(t)= \int_0^t \psi(s)ds,
\end{equation*}
for all $t \in \mathbb{R_+}$, is called the \emph{complementary Young function} of $\Phi$ and satisfies the following basic inequality, known as \emph{Young's inequality}: For every $s,t \in \mathbb{R_+}$,
\begin{equation}\label{Young ineq}
st \leq \Phi(s) + \Psi(t).
\end{equation}

Let $\Phi$ and $\Psi$ be complementary Young functions. The \emph{Orlicz norm} on $L_{\Phi}$ is defined as
\begin{equation}
\|f\|_{L_{\Phi}}'  := \sup \left\lbrace \int_X |f(x)g(x)| d\mu(x) :  \int_X \Psi(|g(x)|)d\mu(x) \leq 1  \right\rbrace.
\end{equation}

The Orlicz norms and gauge norm, (\ref{gauge norm}), are equivalent:
\begin{equation*}
\|f\|_{L_{\Phi}} \leq \|f\|_{L_{\Phi}}' \leq  2\|f\|_{L_{\Phi}}.
\end{equation*}


\subsection{}
Let $(X,\mu)$ and $(Y,\nu)$ be $\sigma$-finite measure spaces with $\mu(X)=\nu(Y)=\infty$. It will be useful in our work to be able to associate to an operator $S: M_+(\mathbb{R_+},m) \rightarrow M_+(\mathbb{R_+},m)$ an operator $\tilde{S}: M_+(X,\mu) \rightarrow M_+(Y,\nu)$, having the property
\begin{equation*}
( \tilde{S}f )^{*_{\nu}}(t) = \left( S f^{*_{\mu}} \right)^{*_{m}}(t),
\end{equation*}
for all $f \in M_+(X,\mu)$ and $t \in \mathbb{R_+}$.  

To do this we require a result on measure-preserving transformations from \cite[page $174$]{Ha50}, which we now describe.

Denote by $Y_{\text{Fin}}$ the class of $\nu$-measurable subsets $E$ of $Y$ with $\nu(E)<\infty$. The functional $d$ defined on $Y_{\text{Fin}} \times Y_{\text{Fin}} $ by
\begin{equation*}
d(E,F)=\nu(E \triangle F), 
\end{equation*}
for all $E,F \in Y_{\text{Fin}}$, is a metric on $Y_{\text{Fin}}$. If the metric space $(Y_{\text{Fin}},d)$ is separable, then the measure space $(Y,\nu)$ is said to be separable. 

It is shown in \cite[page $174$]{Ha50} that, when $(Y,\nu)$ is separable and nonatomic, there exists a $1 \text{-}1$ transformation $\tau$ from $Y_{\text{Fin}}$ onto ${\mathbb{R_+}}_{\text{Fin}}$, with
\begin{equation*}
\tau (E_1-E_2)= \tau(E_1)- \tau(E_2), \ \ \ \tau \left( \bigcup_{n=1}^{\infty} E_n \right)= \bigcup_{n=1}^{\infty} \tau(E_n)
\end{equation*}
and $m(\tau(E))=\nu(E)$, for $E_n,E \in Y_{\text{Fin}}$, $n=1,2,...$ As $(Y,\nu)$ is $\sigma$-finite, $\tau$ can be extended to all $\nu$-measurable subsets of $Y$.

We now state and prove our result concerning $S$ and $\tilde{S}$.
\begin{theorem}\label{construction of Pull back of operator S}
Suppose $(X,\mu)$ and $(Y,\nu)$ are $\sigma$-finite measure spaces, with $\mu(X)=\nu(Y)=\infty$. Assume, in addition, that $(Y,\nu)$ is nonatomic and separable and denote by $\tau$ the measure-preserving transformation between the $\nu$-measurable subsets of $Y$ and the Lebesgue measurable subsets of $\mathbb{R_+}$ described above.

Given the operator $S: M_+(\mathbb{R_+},m) \rightarrow M_+(\mathbb{R_+},m)$, define the operator $\tilde{S}$ at $f \in M_+(X,\mu)$ to be the Radon-Nikodym derivative of the $\nu$-absolutely continuous measure $\lambda$, given at the $\nu$-measurable set $E$ by
\begin{equation*}
\lambda(E)=\int_{\tau(E)} S f^{*_{\mu}}.
\end{equation*}
Then,
\begin{equation}\label{operator Stilde}
(\tilde{S}f)^{*_{\nu}}=\left( Sf^{*_{\mu}} \right)^{*_{m}}, \ \ \ m \text{-a.e.}
\end{equation}
\end{theorem}
\begin{proof}
That $\lambda$ is a measure on the $\nu$-measurable subsets of $Y$ follows from the properties of $\tau$, as does the absolute continuity of $\lambda$ with respect to $\nu$.

Now, (\ref{operator Stilde}) is equivalent to 
\begin{equation}\label{distribution functions action operator Stilde}
m \left( E_{Sf^{*_{\mu}}}(u) \right)= \nu \left( E_{ \tilde{S} f}(u) \right),
\end{equation}
where, for $u \in \mathbb{R_+}$,
\begin{equation*}
E_{Sf^{*_{\mu}}}(u)= \left\lbrace s \in \mathbb{R_+} : \left( Sf^{*_{\mu}}\right)(s)> u  \right\rbrace \ \ \  \text{and} \ \ \  E_{ \tilde{S} f}(u)= \left\lbrace y \in Y : ( \tilde{S} f ) (y)> u  \right\rbrace.
\end{equation*}
We claim that, modulo sets of measure zero,
\begin{equation*}
E_{Sf^{*_{\mu}}}(u)=  \tau \left( E_{ \tilde{S} f}(u) \right),
\end{equation*}
which ensures (\ref{distribution functions action operator Stilde}). Indeed, from
\begin{equation*}
\int_{E} ( \tilde{S}f )(y) d\nu(y)= \lambda (E)= \int_{\tau(E)} S f^{*_{\mu}}(t)dt
\end{equation*}
for all $\nu$-measurable $E \subset E_{ \tilde{S} f}(u)$ and the Lebesgue differentiation theorem, we conclude that, modulo sets of measure zero,
\begin{equation*}
\tau \left( E_{ \tilde{S} f}(u) \right) \subset  E_{Sf^{*_{\mu}}}(u).
\end{equation*}
A similar argument yield,
\begin{equation*}
\tau^{-1} \left( E_{Sf^{*_{\mu}}}(u) \right) \subset E_{ \tilde{S} f}(u),
\end{equation*}
that is,
\begin{equation*}
E_{Sf^{*_{\mu}}}(u) \subset \tau \left( E_{ \tilde{S} f}(u) \right),
\end{equation*}
once again, modulo sets of measure zero.
\end{proof}

\begin{lemma}\label{S-tilda is nu-quasilinear if S is}
Let the measures $\mu$ and $\nu$ and the operators $S$ and $\tilde{S}$ be as in \Cref{construction of Pull back of operator S}. Then, $\tilde{S}$ is $r$-quasilinear on $M_+(X,\mu)$, provided $S$ is monotone and dilation-commuting as well as $r$-quasilinear on $M_+(X,\mu)$.
\end{lemma}
\begin{proof}
Given $f$ and $g$ in $M_+(X,\mu)$ one has, by (\ref{operator Stilde}),
\begin{equation*}
\begin{split}
(\tilde{S}(f+g))^{*_{\nu}}(t) 
& = \left[  S \left( (f+g)^{*_{\mu}} \right) \right]^{*_{m}}(t)\\
& \leq  \left[  S \left( f^{*_{\mu}}({\textstyle\frac{\cdot}{2}})+ g^{*_{\mu}}({\textstyle\frac{\cdot}{2}}) \right) \right]^{*_{m}}(t)\\
& \leq C \left[ \left[   S \left( f^{*_{\mu}}({\textstyle\frac{\cdot}{2}}) \right)    \right]^{*_{m}} (ct)+  \left[   S \left(  g^{*_{\mu}}({\textstyle\frac{\cdot}{2}}) \right)   \right]^{*_{m}} (ct) \right] \\
& = C \left[ \left[   \left( S f^{*_{\mu}} \right) ({\textstyle\frac{\cdot}{2}})    \right]^{*_{m}} (ct)+  \left[   \left( S g^{*_{\mu}} \right) ({\textstyle\frac{\cdot}{2}})    \right]^{*_{m}} (ct) \right] \\
& = C \left[ \left(  S f^{*_{\mu}}   \right)^{*_{m}} ({\textstyle\frac{c}{2}}t)   +  \left(  S g^{*_{\mu}}   \right)^{*_{m}} ({\textstyle\frac{c}{2}}t) \right] \\
& = C \left[ (\tilde{S}f)^{*_{\nu}} ({\textstyle\frac{c}{2}}t)  +  (\tilde{S}g)^{*_{\nu}}  ({\textstyle\frac{c}{2}}t)     \right].
\end{split}
\end{equation*}
\end{proof}

Recall that given indices $p_0,p_1,r_0$ and $r_1$ satisfying $1<p_0<p_1<\infty$, $1 \leq r_1,r_2 < \infty$, we say that an $r$-quasilinear operator $T : S(X,\mu) \rightarrow M(Y,\nu)$ is in the class $W( (p_0,r_0),(p_1,r_1);\mu,\nu )$ if
\begin{equation*}
\| Tf \|_{L_{p_i,\infty}(Y,\nu)} \leq C_{i}  \| f \|_{L_{p_i,r_i}(X,\mu)}
\end{equation*}
where $C_i >0$ is independent of $f \in S(x,\mu)$, $i=1,2$. Such weak-type $(p,r)$ inequalities are equivalent to those in (\ref{type (p,r)}).

We will be concerned with the action of the operators $T \in W( (p_0,r_0),(p_1,r_1);\mu,\nu )$ on Orlicz spaces
\begin{equation*}
L_{\Phi}(X,\mu) \subset L_{p_0,r_0}(X,\mu)  + L_{p_1,r_1}(X,\mu).
\end{equation*}
Our next result treats an important decomposition of functions in $L_{p_0,r_0}(X,\mu)  + L_{p_1,r_1}(X,\mu)$.

\begin{lemma}\label{membership of splitted parts}
Fix indices $p_0,p_1,r_0$ and $r_1$ satisfying $1<p_0<p_1<\infty$, $1 \leq r_0, r_1 < \infty$. Let $(X,\mu)$ be a $\sigma$-finite measure space with (for simplicity) $\mu(X)=\infty$ and suppose $f \in \left( L_{p_0,r_0} + L_{p_1,r_1} \right) (X,\mu)$, $t \in \mathbb{R_+}$. At $x \in X$, set
\begin{equation*}
   f_1(x)= \min[|f(x)|, f^*(t)] \cdot \text{sgn}f(x) \ \ \ \text{and} \ \ \  f_0(x)=f(x)-f_1(x). 
\end{equation*}
Then, $f_0 \in L_{p_0,r_0}(X,\mu)$ and $f_1 \in L_{p_1,r_1}(X,\mu)$.
\end{lemma}

\section{Interpolation results for the class $W((p,r),(\infty ,\infty);\mu,\nu)$}\label{Class W(prInfinity)}

In the present section we shall give a description of the interpolation  pairs,  $(L_{\Phi_2}(X,\mu),$  $L_{\Phi_1}(Y,\nu))$, of Orlicz spaces for which every $T \in W((p,r),(\infty,\infty);\mu,\nu)$  maps $L_{\Phi_2}(X,\mu)$ boundedly into $L_{\Phi_1}(Y,\nu)$.
This class  naturally arises, as explained in the Introduction, as an intermediate step in determining the interpolation pairs  $(L_{\Phi_2}(X,\mu), L_{\Phi_1}(Y,\nu))$ for $W((p,r_1),(q,r_2);\mu,\nu)$.


The conditions imposed on an operator $T \in W((p,r),(\infty,\infty);\mu,\nu)$ ensure that it is dominated by a Calder\'on operator, $H^{p,r}$, (in the sense of \cite[page 141]{BS88}), namely,
\begin{equation}\label{joint weak type}
(Tf)^*(t)\leq C \, H^{p,r}f^*(t),
\end{equation}
in which $C>0$ is independent of  $f \in (L_{p,r}+L_{\infty})(X,\mu)$ and $t \in \mathbb{R_+}$.  Moreover, $H^{p,r} \in W((p,r),(\infty,\infty);m,m)$ and it is essentially the smallest operator such that (\ref{joint weak type}) holds.

The fundamental interpolation theorem of Calder\'on \cite{Ca66}, see also \cite[Chapter 3, Theorem 5.7]{BS88}, describes the action of operators satisfying (\ref{joint weak type}), \emph{when $r=1$}, on rearrangement-invariant spaces in terms of the boundedness of the Calderon operator on their representative spaces. In \Cref{Calderon theorem}, we formulate a Calderon-type theorem for the operators of the type in (\ref{joint weak type}). 
Thus, it is enough to characterise those $\Phi_1$ and $\Phi_2$ for which $H^{p,r}$ maps $L_{\Phi_2}(\mathbb{R_+},m)$ boundedly into $L_{\Phi_1}(\mathbb{R_+},m)$. 

As we shall see in \Cref{Hpr dilation invariant},  $H^{p,r}$ is a dilation-commuting operator, and therefore it suffices to work with a modular inequality rather than a norm inequality, as explained in \Cref{EquvlcHprnormandmod}. Using the estimates of the distribution function for $H^{p,r}f^*$ in \Cref{UDEHpr class} and  \Cref{LDEHpr} we are able to reformulate the modular inequality for $H^{p,r}$ involving $\Phi_1$ and $ \Phi_2$ as a weighted Hardy inequality (with weights involving $\phi_1$ and $\phi_2$) on nonnegative, nonincreasing functions. The duality principle of Sawyer \cite{Sw90} then allows us to pass to an equivalent inequality for a Hardy-type operator on nonnegative measurable functions. To conclude, we then arrive at our desired necessary and sufficient conditions by invoking the results of Stepanov, \cite{Stp90} for such inequalities.

Our conditions depend on $r$ for $1\leq r<p$, see \Cref{main result 1 class Wpr modified by Kerman}, and are, interestingly, independent of $r$ for $p \leq r < \infty$, see \Cref{main result 22 class Wpr}.
For $1\leq r<p$, conditions, that we get, can readily be seen as an extension of the earlier results of  
Cianchi \cite{Ci99} for the case $r=1$. For $p \leq r < \infty,$ we get the well-known conditions as given by Zygmund \cite[Theorem 4.22, page 116]{Zy57}  and Stromberg \cite{Str79}.

\subsection{\texorpdfstring{A Calder\'on-type theorem}{}}\label{Calderon type theorem}


The following result is modeled on  Theorem 4.11 in \cite[page 223]{BS88} or Theorem 8 in \cite{Ca66}, and the proof  carries over almost verbatim to this slightly more general case.

\begin{theorem}\label{Hprdominance}
Let $(X,\mu)$ and $(Y,\nu)$ be $\sigma$-finite measure spaces and suppose $1 \leq p,r<\infty$. If $T \in W((p,r),(\infty ,\infty);\mu,\nu) $, then
\begin{equation}\label{eq:Hpr}
(Tf)^*(t)\leq K  H^{p,r}f^*(t),
\end{equation}
where $K>0$ is independent of $f \in (L_{p,r}+L_{\infty})(X,\mu)$ and $t \in \mathbb{R_+}$.

Further, the operator $H^{p,r}$ is in the class $W((p,r),(\infty ,\infty);m,m).$
\end{theorem}

 \begin{proof}
 Let $ f\in(L_{p,r}+L_{\infty})(X,\mu)$ and fix $t>0$.
Set
\begin{equation*}
   f_1(x)= \min[|f(x)|, f^*(t)] \cdot \text{sgn}f(x)
\end{equation*}
and 
\begin{equation*}
f_0(x)=f(x)-f_1(x)=[|f(x)|-f^*(t)]^+  \cdot \text{sgn}f(x),
\end{equation*}
for $x \in X$.  Then, $f= f_0 + f_1$ and for all $s>0$
\begin{equation*}
f^*_0(s)= [f^*(s)-f^*(t)]^+,
\end{equation*}
\begin{equation*}
f_1^*(s)= \min (f^*(s),f^*(t)).
\end{equation*}
Clearly $f_1 \in L_{\infty}(X,\mu)$; moreover, $f_0 \in L_{p,r}(X,\mu)$ follows from \Cref{membership of splitted parts}.

We next establish inequality (\ref{eq:Hpr}). Since $T$ is a $r$-quasilinear operator with, say, constant of $r$-quasilinearity $C>0$ and $0<c<1$, we have for $t > 0$,
\begin{equation*}
\begin{split}
(Tf)^*(t) & = (T(f_0+f_1))^*(t)\\
& \leq C [(Tf_0)^*(ct)+(Tf_1)^*(ct)].
\end{split}
\end{equation*}
Since
\begin{equation*}
T : L_{p,r}(X,\mu) \rightarrow L_{p,\infty}(Y,\nu),
\end{equation*}
it follows that
\begin{equation*}
\begin{split}
(Tf_0)^*(ct) & \leq M_{p,r}  ( ct )^{-\frac{1}{p}} \left( {\textstyle\frac{r}{p}}  \int_{\mathbb{R_+}}(s^{\frac{1}{p}}f_0^*(s))^r \frac{ds}{s} \right)^{\frac{1}{r}} \\
& \leq M_{p,r} c^{-\frac{1}{p}} \left( {\textstyle\frac{r}{p}} \right)^{\frac{1}{r}} \ t^{-\frac{1}{p}} \left( \int_0^{t}(s^{\frac{1}{p}}f^*(s))^r \frac{ds}{s} \right)^{\frac{1}{r}} \\
& = M_{p,r} c^{-\frac{1}{p}}   \left( {\textstyle\frac{r}{p}} \right)^{\frac{1}{r}}  \  H^{p,r}f^*(t).
\end{split}
\end{equation*}
Again, $f^*$ is a decreasing function, so
\begin{equation}\label{eq:Idd}
\begin{split}
H^{p,r}f^*(t)& =  t^{-\frac{1}{p}}\left( \int_0^{t}(s^{\frac{1}{p}}f^*(s))^r \frac{ds}{s} \right)^{\frac{1}{r}}\\
& \geq f^*(t)t^{-\frac{1}{p}}\left( \int_0^{t}s^{\frac{r}{p}-1} ds \right)^{\frac{1}{r}}\\
&= \left( {\textstyle\frac{p}{r}} \right)^{\frac{1}{r}} \ f^*(t).
\end{split}
\end{equation}
Using the fact that
\begin{equation*}
T : L_{\infty}(X,\mu) \rightarrow L_{\infty}(Y,\nu),
\end{equation*}
and (\ref{eq:Idd}), we get
\begin{equation*}
\begin{split}
(Tf_1)^*(ct)& \leq M_{\infty} \parallel f_1 \parallel_{L_{\infty}(X,\mu)}\\ & = M_{\infty} f^*(t)\\
& \leq M_{\infty}  \left( {\textstyle\frac{p}{r}} \right)^{-\frac{1}{r}}  H^{p,r}f^*(t).
\end{split}
\end{equation*}
Combining these yields
\begin{equation*}
\begin{split}
(Tf)^*(t) & \leq C[(Tf_0)^*(ct )+(Tf_1)^*(ct)]\\
& = K \,  H^{p,r}f^*(t),
\end{split}
\end{equation*}
where $K =   \left( M_{p,r} c^{-\frac{1}{p}} +  M_{\infty} \right)  \left(\frac{r}{p}\right)^{\frac{1}{r}} C $.
\end{proof}

We are now in a position to formulate a Calder\'on-type theorem for operators in $W((p,r),(\infty ,\infty);\mu,\nu)$.
\begin{theorem}\label{Calderon theorem}
Fix $p$ and $r$, $1 \leq p,r < \infty$, and suppose $(X,\mu)$ and $(Y,\nu)$ are $\sigma$-finite  measure spaces with $\mu(X)=\nu(Y)=\infty$, the latter being nonatomic and separable. Then, the  following are equivalent:
\begin{enumerate}
\item[(1)] Every operator $T$ in the class $W((p,r),(\infty ,\infty);\mu,\nu)$ is bounded from $L_{\Phi_2}(X,\mu)$ to $L_{\Phi_1}(Y,\nu)$;
\item[(2)] The operator $H^{p,r}$ is bounded from $L_{\Phi_2}(\mathbb{R_+},m)$ to $L_{\Phi_1}(\mathbb{R_+},m)$.
\end{enumerate}
\end{theorem}
\begin{proof}
We first show $(2)$ implies $(1)$. Let $T$ be any operator in the class $W((p,r),(\infty ,\infty);\mu,\nu)$. Then, by \Cref{Hprdominance}, 
\begin{equation}\label{dominace by calderon type operator}
(Tf)^*(t)\leq K  H^{p,r}f^*(t),
\end{equation}
for all $f \in (L_{p,r}+L_{\infty})(X,\mu)$ and for all $t \in \mathbb{R_+}$. Now, from $(2)$, it follows that $L_{\Phi_2}(X,\mu) \subseteq (L_{p,r}+L_{\infty})(X,\mu)$. Indeed, we have,
\begin{align*}
C \| f^* \|_{L_{\Phi_2}(\mathbb{R_+},m)} &\geq \| H^{p,r} f^* \|_{L_{\Phi_1}(\mathbb{R_+},m)} \\
& \geq \frac{1}{2}  \sup \left\lbrace  \int_{\mathbb{R_+}} H^{p,r}f^*(s)g^*(s) ds : \  \|g^*\|_{L_{\Psi_1}(\mathbb{R_+},m)} \leq 1 \right\rbrace .
\end{align*}
Taking $g^*= \chi_{(0,1)} / \| \chi_{(0,1)} \|_{L_{\Psi_1}(\mathbb{R_+},m)} $, we get \[H^{p,r}f^*(1) \leq \int_0^1 H^{p,r}f^*(s)ds \leq D  \| f^* \|_{L_{\Phi_2}(\mathbb{R_+},m)}< \infty, \] with $D = 2C\| \chi_{(0,1)} \|_{L_{\Psi_1}(\mathbb{R_+},m)}$. From the estimate of the K-functional for the pair $\left( L_{p,r}(X,\mu), \right.$ $\left. L_{\infty}(X,\mu) \right)$ (see \cite[Theorem $4.2$]{Ho70}), we have that 
\begin{equation*}
H^{p,r}f^*(1)= \int_0^1 f^*(s)s^{r/p-1}ds \approx \|f \|_{(L_{p,r}+L_{\infty})(X,\mu)} < \infty,
\end{equation*}
for all $f \in (L_{p,r}+L_{\infty})(X,\mu)$.

 Next, $(2)$, together with (\ref{dominace by calderon type operator}), implies that, given $f$ in $L_{\Phi_2}(X,\mu) $ and hence in $(L_{p,r}+L_{\infty})(X,\mu)$, one has
\begin{align*}
\| Tf\|_{L_{\Phi_1}(Y,\nu)}= \| (Tf)^*\|_{L_{\Phi_1}(\mathbb{R_+},m)} & \leq K \| H^{p,r}f^* \|_{L_{\Phi_1}(\mathbb{R_+},m)} \\
& \leq KC \| f^* \|_{L_{\Phi_2}(\mathbb{R_+},m)} = KC \| f\|_{L_{\Phi_2}(X,\mu)},
\end{align*}
so that the operator $T$ is bounded from $L_{\Phi_2}(X,\mu)$ to $L_{\Phi_1}(Y,\nu)$.

Conversely, assume that $(1)$ holds. In \Cref{construction of Pull back of operator S} take $S=H^{p,r}$ and denote by $\tilde{H}^{p,r}$ the operator $\tilde{S}$ guaranteed to exist by that theorem. In particular, then,
\begin{equation*}
(\tilde{H}^{p,r}f)^{*_{\nu}}=  H^{p,r}f^{*_{\mu}}, \ \ \ m \text{-a.e.},
\end{equation*}
for all $f \in M(X,\mu)$, since $\left( H^{p,r}f^{*_{\mu}} \right)(t)= \left( \int_0^{1} f^*(ts)s^{{\textstyle\frac{r}{p}-1}}ds \right)^{\frac{1}{r}}$ is nonincreasing, so $\left( H^{p,r}f^{*_{\mu}} \right)^{*}= H^{p,r}f^{*_{\mu}}$. Moreover, since $H^{p,r}$ is $r$-quasilinear, $\tilde{H}^{p,r}$ will, according to \Cref{S-tilda is nu-quasilinear if S is}, be $r$-quasilinear.

Next,
\begin{equation*}
\tilde{H}^{p,r} : L_{p,r}(X,\mu) \rightarrow L_{p,\infty}(Y,\nu) \ \ \  \text{and} \ \ \  \tilde{H}^{p,r} : L_{\infty}(X,\mu) \rightarrow L_{\infty}(Y,\nu)
\end{equation*}
boundedly.

Indeed, from \Cref{Hprdominance}, $H^{p,r} : L_{p,r}(\mathbb{R_+},m) \rightarrow L_{p,\infty}(\mathbb{R_+},m)$, so given $f \in L_{p,r}(X,\mu)$, one has
\begin{equation*}
\begin{split}
\| \tilde{H}^{p,r} f \|_{L_{p,\infty}(Y,\nu)}
& = \| ( \tilde{H}^{p,r} f )^{*_{\nu}} \|_{L_{p,\infty}(\mathbb{R_+},m)}\\
& = \|  H^{p,r} f^{*_{\mu}}  \|_{L_{p,\infty}(\mathbb{R_+},m)} \\
& \leq C \|   f^{*_{\mu}}  \|_{L_{p,r}(\mathbb{R_+},m)} \\
& = C \|   f  \|_{L_{p,r}(X,\mu)},
\end{split}
\end{equation*}
that is, $\tilde{H}^{p,r} : L_{p,r}(X,\mu) \rightarrow L_{p,\infty}(Y,\nu)$ boundedly. Similarly, $\tilde{H}^{p,r} : L_{\infty}(X,\mu) \rightarrow L_{\infty}(Y,\nu)$ boundedly. We have now shown $\tilde{H}^{p,r} \in   W((p,r),(\infty ,\infty);\mu,\nu)$, whence, by $(1)$, $\tilde{H}^{p,r} : L_{\Phi_2}(X,\mu) \rightarrow L_{\Phi_1}(Y,\nu)$  boundedly.

In \Cref{construction of Pull back of operator S}, take $X$ to be $\mathbb{R_+}$, $\mu$ to be $m$, $Y$ to be $X$, $\nu$ to be $\mu$ and $S$ to be the operator $g \rightarrow g^{*_{m}}$. Given $f \in M_+(\mathbb{R_+},m)$, set $\tilde{f}= \tilde{S}f  \in M_+(X,\mu)$, so that $\tilde{f}^{*_{\mu}}= f^{*_{m}}$.

Thus,
\begin{align*}
\| H^{p,r} f  \|_{L_{\Phi_1}(\mathbb{R_+},m)}
& \leq \| H^{p,r} f^{*_{m}}  \|_{L_{\Phi_1}(\mathbb{R_+},m)}\\
& = \| H^{p,r} \tilde{f}^{*_{\mu}}  \|_{L_{\Phi_1}(\mathbb{R_+},m)} \\
& = \| ( H^{p,r} \tilde{f}^{*_{\mu}})^{*_{m}}  \|_{L_{\Phi_1}(\mathbb{R_+},m)} \\
& = \|( \tilde{H}^{p,r} {\tilde{f}} )^{*_{\nu}}  \|_{L_{\Phi_1}(\mathbb{R_+},m)} \\
& = \| \tilde{H}^{p,r} {\tilde{f}}   \|_{L_{\Phi_1}(Y,\nu)}\\
& \leq C  \|  {\tilde{f}}   \|_{L_{\Phi_2}(X,\mu)}\\
& = C  \|  \tilde{f}^{*_{\mu}}   \|_{L_{\Phi_2}(\mathbb{R_+},m)}\\
& = C  \|  f^{*_{m}}   \|_{L_{\Phi_2}(\mathbb{R_+},m)},
\end{align*}
whence $H^{p,r} : L_{\Phi_2}(X,\mu) \rightarrow  L_{\Phi_1}(Y,\nu)$ boundedly.

This completes the proof.
\end{proof}

\subsection{\texorpdfstring{The Calder\'on  operator $H^{p,r}$ and an associated Hardy inequality}{}}\label{Second section}
 Our next result shows that it is enough to work with the modular inequality for $H^{p,r}$.

\begin{lemma}\label{Hpr dilation invariant}
Let $1 \leq p< \infty$ and $1 \leq r < \infty$. Then,  $H^{p,r}$ is a dilation-commuting operator.
\end{lemma}

 \begin{proof} The proof is an easy exercise in change of variable, hence we omit it.
 \end{proof}

\begin{theorem}\label{EquvlcHprnormandmod}
Let $\Phi_1$ and $\Phi_2$ be Young functions. For $1 \leq p,r < \infty$,  we have that the norm inequality 
\begin{equation}\label{eq:normm}
\parallel H^{p,r}f \parallel _{L_{\Phi_1}(\mathbb{R_+},m)} \leq C \parallel f \parallel_{L_{\Phi_2}(\mathbb{R_+},m)},
\end{equation}
holds for all $f$ in $M_+(\mathbb{R_+},m)$ if and only if the modular inequality
\begin{equation}\label{eq:modd}
  \int_{\mathbb{R_+}}\Phi_1(\left( H^{p,r}f^* \right)(t))dt \leq \int_{\mathbb{R_+}}\Phi_2(Kf^*(s))ds ,
\end{equation}
holds for all $f$ in $M_+(\mathbb{R_+},m)$.
\end{theorem}

 \begin{proof} The proof follows from \cite[Theorem A]{KRS17}, as the norm and the modular inequalities are equivalent for a dilation-commuting operator.
  \end{proof}

We now seek an expression equivalent to the distribution function of $H^{p,r}g$, when $g$ is nonnegative and nonincreasing on $\mathbb{R_+}$.

\begin{lemma}\label{UDEHpr class}
Let $1 \leq p< \infty,1 \leq r < \infty$ and suppose $T \in W((p,r),(\infty,\infty);\mu,\nu)$. Then, for every $f$ in the domain of $T$,
\begin{equation}\label{eq:DE}
 \nu_{Tf}(t) 
\leq {\textstyle\frac{1}{c}} \left( 2^{1+1/r}r^{1/r} C M_{p,r} \right)^p  \frac{1}{t^p}\left( \int_{t/4CM_{\infty}}^{\infty} \mu_{f}(s)^{r/p}s^{r-1}ds \right)^{p/r},
\end{equation}
where $C$ and $c$ are the constant of $r$-quasilinearity of $T$ and $M_{p,r}, M_{\infty}$ are the operator norms in $T: L_{p,r}(X,\mu) \rightarrow L_{p,\infty}(Y,\nu)$ and $T: L_{\infty}(X,\mu) \rightarrow L_{\infty}(Y,\nu)$, respectively.
\end{lemma}

\begin{proof}
Fix $t>0$ and $f \in L_{p,r}+L_{\infty}(X,\mu)$. Let  $k$ be any positive number. Write $f=f_t + f^t$, with
\begin{equation*}
    f^t(x)= \begin{cases}
                 f(x), & |f(x)|>{ \textstyle\frac{t}{2Ck} },\\
                 0,  & |f(x)| \leq { \textstyle\frac{t}{2Ck} },
                \end{cases}
\end{equation*}
and $f_t(x)= f(x)- f^t(x)$. 
Observe that the distribution functions of $f_t$ and $f^t$ are as follows:
\begin{equation*}
    \mu_{f_t}(s)= \begin{cases}
                 \mu_{f}(s) - \mu_{f}({ \textstyle\frac{t}{2Ck} }), & s<{ \textstyle\frac{t}{2Ck} },\\
                 0,  & s \geq { \textstyle\frac{t}{2Ck} }
                \end{cases}
\end{equation*}
and
\begin{equation*}
    \mu_{f^t}(s)= \begin{cases}
                 \mu_{f}({ \textstyle\frac{t}{2Ck} }), & s < { \textstyle\frac{t}{2Ck} },\\
                 \mu_{f}(s),  & s \geq { \textstyle\frac{t}{2Ck} }.
                \end{cases}
\end{equation*}
Then, from the $r$-quasilinearity of $T$
\begin{equation*}
\nu_{Tf}(t) \leq {\textstyle\frac{1}{c}} \left[ \nu_{Tf^t}\left({ \textstyle\frac{t}{2C} } \right)+\nu_{Tf_t}\left( { \textstyle\frac{t}{2C} } \right) \right].
\end{equation*}
For any $t>0$ and $x$ such that $|Tf_t(x)|> { \textstyle\frac{t}{2C} }$, we have  ${ \textstyle\frac{t}{2C} } < |Tf_t(x)| \leq M_{\infty}\|f_t\|_{L_{\infty}(X,\mu)} \leq  M_{\infty} \frac{t}{2Ck}$. Therefore, $\nu_{Tf_t}({ \textstyle\frac{t}{2C} })=0$ when $k \geq M_{\infty}$. So, for such a $k$,
\begin{equation*}
\nu_{Tf}(t) \leq {\textstyle\frac{1}{c}} \nu_{Tf^t}\left({ \textstyle\frac{t}{2C} } \right).
\end{equation*}
Since $T : L_{p,r}(X,\mu) \rightarrow L_{p,\infty}(Y,\nu)$, with operator norm, say, $M_{p,r}$, we have, for any $y>0$,
\begin{equation*}
\begin{split}
y \ \nu_{Tf^t}(y)^{\frac{1}{p}} & \leq M_{p,r} \|f^t\|_{L_{p,r}}\\
& = r^{1/r} M_{p,r}   \left( \int_0^{\infty} \mu_{f^t}(s)^{r/p}s^{r-1}ds \right)^{1/r}\\
& =r^{1/r} M_{p,r}\left( \mu_{f}\left( { \textstyle\frac{t}{2Ck} } \right)^{r/p} {\textstyle\frac{ \left( \textstyle\frac{t}{2Ck} \right)^r}{r} }+ \int_{\textstyle\frac{t}{2Ck}}^{\infty} \mu_{f}(s)^{r/p}s^{r-1}ds \right)^{1/r}\\
& \leq r^{1/r} M_{p,r} \left( 2 \int_{\textstyle\frac{t}{4Ck}}^{\infty} \mu_{f}(s)^{r/p}s^{r-1}ds \right)^{1/r}.
\end{split}
\end{equation*}
Indeed, for any $x>0$
\begin{equation*}
\begin{split}
 \int_{x/2}^{\infty} \mu_{f}(s)^{r/p}s^{r-1}ds & =  \int_{x/2}^{x} \mu_{f}(s)^{r/p}s^{r-1}ds+  \int_{x}^{\infty} \mu_{f}(s)^{r/p}s^{r-1}ds\\
 & \geq \mu_{f}(x)^{r/p}{\textstyle\frac{x^r}{r}} \left( 1- { \textstyle\frac{1}{2^r}} \right) + \int_{x}^{\infty} \mu_{f}(s)^{r/p}s^{r-1}ds\\
 & \geq {\textstyle\frac{1}{2}}  \left(    \mu_{f}(x)^{r/p}  { \textstyle\frac{x^r}{r}}  + \int_{x}^{\infty} \mu_{f}(s)^{r/p}s^{r-1}ds    \right),
\end{split}
\end{equation*}
which yields the assertion on taking $x={ \textstyle\frac{t}{2Ck} }$. Again, with $y={ \textstyle\frac{t}{2C} }$, we get
\begin{equation*}
{ \textstyle\frac{t}{2C} } \ \nu_{Tf^t}\left( { \textstyle\frac{t}{2C} } \right)^{\frac{1}{p}} 
\leq r^{1/r} M_{p,r} \left( 2 \int_{\frac{t}{4Ck}}^{\infty} \mu_{f}(s)^{r/p}s^{r-1}ds \right)^{1/r},
\end{equation*}
which implies
\begin{equation*}
 \nu_{Tf}(t) 
\leq {\textstyle\frac{1}{c}} \left( \textstyle\frac{ M_{p,r}(2r)^{1/r}}{{ \textstyle\frac{t}{2C} }} \right)^p \left( \int_{\frac{t}{4Ck}}^{\infty} \mu_{f}(s)^{r/p}s^{r-1}ds \right)^{p/r}.
\end{equation*}
\end{proof}
 

 \begin{theorem}\label{LDEHpr}
Fix $p$ and $r$, with $1<p<\infty$ and $1 \leq r< \infty$. Then, for any nonnegative, nonincreasing $g \in (L_{p,r}+L_{\infty})(\mathbb{R_+},m)$ and $t \in \mathbb{R_+}$, one has
\begin{align}\label{equivalent expression for distribution function of Hprf}
 p^{p/r}     \left(  \int_{t}^{\infty}  m_{g^*}(s)^{r/p} s^{r-1}ds \right)^{p/r}      &    \leq      t^p \,  m_{H^{p,r}g^*}(t)  \nonumber\\
&   \leq     2^{ 2p+ 1 } p^{p/r}   \left(  \int_{  \frac{t}{       2^{3-\frac{1}{r}}  \left( {\frac{p}{r}} \right)^{\frac{1}{r}}            }    }^{\infty}  m_{g^*}(s)^{r/p} s^{r-1}ds \right)^{p/r} 
\end{align}
\end{theorem}
\begin{proof}
The operator $H^{p,r}$ is in $W((p,r),(\infty,\infty);m,m)$, with $r$-quasilinearity constants $C=2^{1-\frac{1}{r}}$, $c={\textstyle\frac{1}{2}}$ and $M_{p,r},M_{\infty}$ are less than or equal to $\left( \frac{p}{r} \right)^{\frac{1}{r}}$. Also, observe that
\begin{equation*}
H^{p,r}f \leq H^{p,r}f^*, \ f \in M_+(\mathbb{R_+},m).
\end{equation*}
This suffices to establish the first of the inequalities in (\ref{equivalent expression for distribution function of Hprf}), in view of \Cref{UDEHpr class}.

To prove the first inequality we begin by letting $\tau_0$ be the least $\tau$ for which $\left( H^{p,r}g \right)(\tau)=t$. Then,
\begin{equation}
\tau_0 = m_{H^{p,r}g}(t)
\end{equation}
and
\begin{equation}
\left( H^{p,r}g \right)(\tau_0)= t \Leftrightarrow \left( \tau_0^{-r/p} \int_0^{\tau_0} g(s)^{r} s^{r/p-1}ds \right)^{1/r}= t.
\end{equation}
Since $H^{p,r}g(t) \geq g(t)$, implying thereby $\tau_0 = m_{H^{p,r}g}(t) \geq m_{g}(t)$,
\begin{align*}
m_{H^{p,r}g}(t) = \tau_0 & = \frac{1}{t^{p}} \left(  \int_0^{\tau_0} g(s)^{r} s^{r/p-1}ds \right)^{\frac{p}{r}}\\
& \geq \frac{1}{t^{p}} \left(  \int_0^{m_{g}(t)} g(s)^{r} s^{r/p-1}ds \right)^{\frac{p}{r}}\\
& = \frac{1}{t^{p}} \left(  \int_{\mathbb{R_+}} ( \chi_{(0, m_{g}(t))} g)(s)^{r} s^{r/p-1}ds \right)^{\frac{p}{r}}\\
& = {\left( \textstyle\frac{p}{r} \right)}^{\frac{p}{r}}  \frac{1}{t^{p}}    \left(  {\textstyle\frac{r}{p}} \int_{\mathbb{R_+}} ( \chi_{(0, m_{g}(t))} g)^*(s)^{r} s^{r/p-1}ds \right)^{\frac{p}{r}}\\
& =    {\left( \textstyle\frac{p}{r} \right)}^{\frac{p}{r}}  \frac{1}{t^p} \| \chi_{(0, m_{g}(t))} g \|_{L_{p,r}(\mathbb{R_+})}^p \\
& = p^{\frac{p}{r}} \frac{1}{t^{p}} \left(  \int_{\mathbb{R_+}}  m_{[\chi_{(0, m_{g}(t))} g]}(s)^{r/p} s^{r-1}ds \right)^{\frac{p}{r}}\\
& = p^{\frac{p}{r}}  \frac{1}{t^{p}} \left( \int_0^{t} m_{g}(s)^{r/p} s^{r-1}ds +  \int_{t}^{\infty}  m_{g}(s)^{r/p} s^{r-1}ds \right)^{\frac{p}{r}}
\end{align*}
where in the last but one equality integral is realized as a Lorentz space norm, so finally  we get
\begin{equation*}
m_{H^{p,r}g}(t) \geq     p^{p/r} \frac{1}{t^{p}} \left(  \int_{t}^{\infty}  m_{g}(s)^{r/p} s^{r-1}ds \right)^{p/r}.
\end{equation*}
\end{proof}

Using \Cref{LDEHpr} we can reduce a modular inequality (and hence the equivalent norm inequality) involving $H^{p,r}$ to a weighted Hardy inequality.
\begin{theorem}\label{Ineq equi Mod Hpr}
Fix $p$ and $r$, where $1 < p<\infty$ and $ 1 \leq r < \infty$. Suppose $\Phi_i(t)=\int_0^t \phi_i(s)ds$, $i=1,2$ are Young functions. Then, the following are equivalent:  
\begin{enumerate}
\item
 There exists a constant $C>0$ such that
\begin{equation}\label{eq:modular inequality for Hpr}
  \int_{\mathbb{R_+}}\Phi_1 \left( (H^{p,r}f^*)(t) \right) dt \leq \int_{\mathbb{R_+}}\Phi_2(Cf^*(s))ds ,
\end{equation}
holds for all $f$ in $M_+(\mathbb{R_+},m)$.;
\item
There exist $C_1,C_2>0$, such that the weighted Hardy inequality
\begin{align}\label{eq:HD}
 \int_{\mathbb{R_+}}    \left( \int_{x}^{\infty} g(s)ds \right)^{p/r} { \textstyle\frac{\phi_1(x^{\frac{1}{r}})}{x^{\frac{p-1}{r}+1}  } } dx 
\leq C_1  \int_{\mathbb{R_+}} g(y)^{p/r}   \phi_2 \left( C_2 y^{\frac{1}{r}} \right) y^{\frac{1}{r}-1}dy,
\end{align}
 holds for all nonnegative, nonicreasing function $g$ on $\mathbb{R_+}$.
\end{enumerate}
Moreover, $C_2= \frac{C}{4KM_{\infty}} $ and  $C_1 = \frac{r}{p}2^{r-1}\frac{M_{\infty}^{r/p'}}{M_{p,r}^r} \left( \frac{C}{4K} \right)^{r/p}$, where $k=1$ and $K$ are the constants of $r$-quasilinearity for the operator $H^{p,r}$ and $M_{p,r}, M_{\infty}$ are operator norms of $H^{p,r}: L_{p,r}(\mathbb{R_+},m) \rightarrow L_{p,\infty}(\mathbb{R_+},m)$ and $H^{p,r}: L_{\infty}(\mathbb{R_+},m) \rightarrow L_{\infty}(\mathbb{R_+},m)$ respectively.
\end{theorem}
\begin{proof}
Let $\alpha =2^{2+\frac{1}{p}} p^{\frac{1}{r}}$ and $\beta = 2^{3-\frac{1}{r}} \left( \frac{p}{r} \right)^{\frac{1}{r}}$. Then, in view of \Cref{LDEHpr},
\begin{equation*} 
\begin{split}
\int_{\mathbb{R_+}}\Phi_1((H^{p,r}f^*)(t))dt
& = \int_{\mathbb{R_+}} \phi_1(t)m_{H^{p,r}f^*}(t)dt \\
& \leq \alpha^p \int_{\mathbb{R_+}}    \frac{\phi_1(t)}{t^p}\left( \int_{t/\beta}^{\infty} m_{f^*}(s)^{r/p}s^{r-1}ds \right)^{p/r}dt \\
& = \left( \frac{\alpha}{\beta} \right)^p \int_{\mathbb{R_+}}    \frac{\phi_1(t)}{t^p}\left( \int_{t}^{\infty} m_{\beta f^*}(s')^{r/p}s'^{r-1}ds' \right)^{p/r}dt \\
& = \left( \frac{\alpha}{\beta} \right)^p  \frac{1}{r^{p/r}} \int_{\mathbb{R_+}}    \frac{\phi_1(t)}{t^p}\left( \int_{t^r}^{\infty} m_{\beta f^*}(y^{\frac{1}{r}})^{r/p} dy \right)^{p/r}dt \\
& = \left( \frac{\alpha}{\beta} \right)^p  \frac{1}{r^{1+p/r}} \int_{\mathbb{R_+}}    \frac{\phi_1(x^{\frac{1}{r}})}{x^{p/r}}\left( \int_{x}^{\infty} m_{\beta f^*}(y^{\frac{1}{r}})^{r/p} dy \right)^{p/r}  x^{\frac{1}{r}-1} dx \\
& = \left( \frac{\alpha}{\beta} \right)^p  \frac{1}{r^{1+p/r}} \int_{\mathbb{R_+}}    \left( \int_{x}^{\infty} m_{\beta f^*}(y^{\frac{1}{r}})^{r/p} dy \right)^{p/r} \frac{\phi_1(x^{\frac{1}{r}})}{x^{\frac{p-1}{r}+1}} dx 
\end{split}
\end{equation*}
Now, given (\ref{eq:HD}), the latter will be
\begin{align*}
& \leq \left( \frac{\alpha}{\beta} \right)^p  \frac{1}{r^{1+p/r}} C_1 \int_{\mathbb{R_+}}      m_{\beta f^*}(y^{\frac{1}{r}}) \, \phi_2 \left( C_2 y^{\frac{1}{r}} \right) y^{\frac{1}{r}-1}dy \\
& \leq \left( \frac{\alpha}{\beta} \right)^p  \frac{1}{r^{p/r}} C_1 \int_{\mathbb{R_+}}      m_{\beta f^*}(s) \, \phi_2 \left( C_2 s \right) ds\\
& \leq \left( \frac{\alpha}{\beta} \right)^p  \frac{1}{r^{p/r}} \frac{C_1}{C_2} \int_{\mathbb{R_+}}     \Phi_2 \left( \beta C_2 f^*(s) \right) ds\\
& \leq  \int_{\mathbb{R_+}}     \Phi_2 \left( C f^*(s) \right) ds,
\end{align*}
$C= \max \left[ \beta C_2,  \left( \frac{\alpha}{\beta} \right)^p  \frac{1}{r^{p/r}} \frac{C_1}{C_2}      \beta C_2 \right] $. Thus, (\ref{eq:HD}) implies (\ref{eq:modular inequality for Hpr}).

Suppose, next, that (\ref{eq:modular inequality for Hpr}) holds. The nonnegative, nonincreasing $g$ in (\ref{eq:HD}) is of the form $m_{\beta f^*}(y^{\frac{1}{r}})^{r/p}$ for some $f^*$, namely, for
\begin{equation*}
f^*(t) =  \frac{   \left[ g(s^r)^{p/r} \right]^{-1}(t)     }{\beta}, \ t \in \mathbb{R_+}.
\end{equation*}
So, (\ref{eq:HD}) is equivalent to the inequality
\begin{align*}
 \int_{\mathbb{R_+}}    \left( \int_{x}^{\infty} m_{\beta f^*}(y^{\frac{1}{r}})^{r/p} dy \right)^{p/r} { \textstyle\frac{\phi_1(x^{\frac{1}{r}})}{x^{\frac{p-1}{r}+1}  } }  dx 
\leq C_1  \int_{\mathbb{R_+}} m_{\beta f^*}(y^{\frac{1}{r}})   \phi_2 \left( C_2 y^{\frac{1}{r}} \right) y^{\frac{1}{r}-1}dy.
\end{align*}
Taking $x=t^r$ in the first integral we get
\begin{align*}
 \int_{\mathbb{R_+}}    \left( \int_{t^r}^{\infty} m_{\beta f^*}(y^{\frac{1}{r}})^{r/p} dy \right)^{p/r} { \textstyle\frac{\phi_1(t)}{t^p  } } dt
\leq \frac{C_1}{r}  \int_{\mathbb{R_+}} m_{\beta f^*}(y^{\frac{1}{r}})   \phi_2 \left( C_2 y^{\frac{1}{r}} \right) y^{\frac{1}{r}-1}dy.
\end{align*}
Again, with $y=s^r$ in either side of this last inequality we arrive at
\begin{align*}
 \int_{\mathbb{R_+}}    \left( \int_{t}^{\infty} m_{\beta f^*}(s)^{r/p} s^{r-1}ds \right)^{p/r} { \textstyle\frac{\phi_1(t)}{t^p  } } dt
\leq \frac{C_1}{r^{p/r}}  \int_{\mathbb{R_+}} m_{\beta f^*}(s)   \phi_2 \left( C_2 s \right) ds.
\end{align*}
In view of (\ref{equivalent expression for distribution function of Hprf}) and (\ref{eq:modular inequality for Hpr}) we  have that
\begin{align*}
 \int_{\mathbb{R_+}}    \left( \int_{t}^{\infty} m_{\beta f^*}(s)^{r/p} s^{r-1}ds \right)^{p/r} { \textstyle\frac{\phi_1(t)}{t^p  } } dt
& \leq  p^{-p/r} \int_{\mathbb{R_+}} m_{H^{p,r}(\beta f^*)}(t) \, \phi_1(t)dt\\
& =  p^{-p/r} \int_{\mathbb{R_+}} \Phi_1 \left( H^{p,r} ( \beta f^* ) (t) \right) dt \\
& \leq  p^{-p/r} \int_{\mathbb{R_+}} \Phi_2 \left( C \beta f^*  (t) \right) dt \\
& =  {\textstyle\frac{C}{p^{p/r}}} \int_{\mathbb{R_+}} m_{\beta f^*}(s) \phi_2 (Cs)ds.
\end{align*}
So, if we choose $C_1 =  \left( \frac{p}{r} \right)^{-p/r}C$ and $C_2=C$, we have (\ref{eq:HD}) is implied by (\ref{eq:modular inequality for Hpr}).
\end{proof}

In the next two sections, we will be taking up the two-weight Hardy inequality (\ref{eq:HD}) on  nonnegative, nonincreasing functions with weights being functions involving  $\Phi_1$ and $\Phi_2$. As in the classical case, the inequality (\ref{eq:HD}) needs to be studied in two cases depending on whether $\frac{p}{r}>1$ or $\frac{p}{r} \leq 1$. 

\subsection{\texorpdfstring {The case $1 \leq r < p$}{}}\label{The case one less equal r less p}
The dual Hardy operator, $ g \mapsto (Qg)(y):= \int_y^{\infty} g^*(s)ds$, in (\ref{eq:HD}) is an example of a kernel operator, namely, an operator $T$ of the form
\begin{equation*}
Tf(x)= \int_{\mathbb{R_+}}K(x,y)f(y)dy,
\end{equation*}
in  which $f \in M_+(\mathbb{R_+},m)$, $x \in \mathbb{R_+}$ and the nonnegative kernel $K(x,y) \in M(\mathbb{R_+} \times \mathbb{R_+}, m \times m)$.

The following result of E. T. Sawyer \cite{Sw90} reduces the study of a weighted norm inequality for such a $T$ on nonnegative nonincreasing functions on $\mathbb{R_+}$, as in (\ref{eq:HD}), to one on nonnegative functions in $M(\mathbb{R_+},m)$.

\begin{theorem}[E. T. Sawyer, {\cite{Sw90}}]\label{Sw}
Fix $p_1$ and $q_1$, $1< p_1,q_1 < \infty$, and suppose $w(x)$ and $v(x)$ are weights on $\mathbb{R_+}$.  Then, the inequality
\begin{equation}\label{eq:TND}
 \left( \int_{\mathbb{R_+}} Tf(x)^{q_1}w(x)dx \right)^{1/q_1} \leq C \left( \int_{\mathbb{R_+}} f(x)^{p_1}v(x)dx \right)^{1/p_1}
\end{equation}
holds with $C>0$ independent of the nonnegative and nonincreasing function $f$ on $\mathbb{R_+}$ if and only if
\begin{multline}\label{eq:T*N}
 \left( \int_{\mathbb{R_+}}  \textstyle\left( \int_0^x T^*g \right)^{p_1'}  \frac{v(x)}{ \left( \int_0^x v  \right)^{p_1'}   }    dx \right)^{1/p_1'}       +       \frac{      \left( \int_{\mathbb{R_+}} T^*g \right)                 }{   V(\infty)^{1/p_1'}}      \leq C \left( \int_{\mathbb{R_+}} g(x)^{q_1'}w(x)^{1-q_1'}dx \right)^{1/q_1'},
\end{multline}
where $C>0$ does not depend on nonnegative  $g$ in $M(\mathbb{R_+})$. Here $T^*$ is the adjoint of $T$ given by \[T^*g(y)= \int_{\mathbb{R_+}}K(z,y)g(z)dz, \ \ \text{for all} \ y \in \mathbb{R_+}, \] \[ V(x)= \int_0^x v, \ \  \text{for all} \  x \in \mathbb{R_+}\]  and \[ V(\infty)= \lim_{x \rightarrow \infty} V(x).\]
\end{theorem}

The inequality (\ref{eq:HD}) can now be rephrased as (\ref{eq:TND}) with $ (Tg)(x)= \int_x^{\infty} g^*(s)ds$, the dual Hardy operator, $p_1=q_1= p/r>1$, $w(y)= \frac{ \phi_1(y^{\frac{1}{r}}) }{ y^{ \frac{p-1}{r}+1 } }$ and  $v(x)= \frac{ \phi_2(C_2x^{\frac{1}{r}} )}{x^{1-\frac{1}{r}}} $.

With this,
\[ V(x) = \int_0^{x} \phi_2(C_2y^{\frac{1}{r}} )y^{\frac{1}{r}-1}dy = \frac{r}{C_2}\int_0^{C_2x^{\frac{1}{r}}} \phi_2(y) dy= \frac{r}{C_2} \Phi_2(C_2x^{\frac{1}{r}})  \]
and 
\[\int_0^x T^*g = \int_0^x \left[ \int_0^y g(t)dt \right]dy = \int_0^x (x-y)g(y)dy = (I_2g)(x),\]
is the Riemann-Liouville fractional integral operator of order $2$ .

Since
\[V(\infty)= \lim_{x \rightarrow \infty} V(x)= \lim_{x \rightarrow \infty} \Phi_2(C_2x^{\frac{1}{r}})= \infty, \]
(\ref{eq:T*N}) amounts to the inequality
\begin{equation}\label{eq:1rp main ineq}
\left( \int_0^{\infty} I_2g(x)^{p_1'}  \frac{v(x)}{ V(x)^{p_1'}   }    dx \right)^{1/p_1'}    \leq C \left( \int_0^{\infty} g(x)^{q_1'}w(x)^{1-q_1'}dx \right)^{1/q_1'}
\end{equation}
for $0 \leq g \in M(\mathbb{R_+},m)$, with $p_1, w, v$ and $V$ as specified above.


Now, a special case of the main result in Stepanov \cite{Stp90}, asserts that (\ref{eq:1rp main ineq}) holds if and only if for all $x \in \mathbb{R_+}$
\begin{equation}\label{eq:condition1 modified by Kerman}
 \left( \int_t^{\infty} (y-t)^{p_1'}  \frac{v(y)}{V(y)^{p_1'}}   dy \right)^{\frac{1}{p_1'}}  \left( \int_0^{t}  w(y) dy \right)^{\frac{1}{p_1}}< \infty,
 \end{equation}
 and
\begin{equation}\label{eq:condition2 modified by Kerman}
 \left( \int_t^{\infty}  \frac{v(y)}{V(y)^{p_1'}} dy \right)^{\frac{1}{p_1'}}  \left( \int_0^{t} (t-y)^{p_1} w(y) dy \right)^{\frac{1}{p_1}}< \infty,
 \end{equation}
Making the change of variable $x \rightarrow x^{\frac{1}{r}}$ and $y \rightarrow y^{\frac{1}{r}}$,  in the expressions for $w,v$ and $V$, we arrive at the conditions
%
\begin{align}\label{eq:CONDITIONS 1,2 modified by Kerman}
\begin{split}
& \left(\int_{x}^{\infty} (y^r-x^r)^{p_1'}   \frac{\phi_2(C_2y)}{\Phi_2(C_2y)^{p_1'}} dy \right)^{\frac{1}{p_1'}}  \left( \int_0^{x}\frac{ \phi_1(y) }{y^{p}}dy \right)^{\frac{1}{p_1}} \leq A,\\
&   \left(\int_{x}^{\infty}   \frac{\phi_2(C_2y)}{\Phi_2(C_2y)^{p_1'}} dy \right)^{\frac{1}{p_1'}}
 \left( \int_0^{x} (x^r-y^r)^{p_1}  \frac{ \phi_1(y) }{y^{p}}dy \right)^{\frac{1}{p_1}} \leq A. 
\end{split}
\end{align}


The operator $H^{p,1}$, given at $f \in M_+(\mathbb{R_+},m)$ by

\[H^{p,1}f(t)= t^{-\frac{1}{p}}\int_0^t f(s)s^{\frac{1}{p}-1}ds,\]
was found by A. Cianchi, \cite{Ci99}, to satisfy the norm inequality

\begin{equation}\label{eq: norm inequality for Hp1}
\|H^{p,1}f\|_{L_{\Phi_1}(\mathbb{R_+},m)} \leq C \|f\|_{L_{\Phi_1}(\mathbb{R_+},m)}
\end{equation}
if and only if there exist constants $D,B>0$ such that for all $x \in \mathbb{R_+}$,
\begin{equation}\label{eq: Cianchi condition for Hp1}
\left( \int_x^{\infty} \frac{\phi_2(Dy)}{\Phi_2(Dy)^{p'}}y^{p'} dy \right)^{\frac{1}{p'}}  \left( \int_0^{x}\frac{ \phi_1(y) }{y^{p}}dy \right)^{\frac{1}{p}} \leq B.
\end{equation}
This suggests the possibility that the two conditions for the norm boundedness of $H^{p,r}, 1 \leq r < p$ in Theorem 4.3.2 can be replaced by a single condition. That this is the case is the content of

\begin{theorem}\label{main result 1 class Wpr modified by Kerman}
Let $(X,\mu)$ and $(Y,\nu)$ be $\sigma$-finite  measure spaces with $\mu(X)=\nu(Y)=\infty$, the latter being nonatomic and separable. Fix the indices $p$ and $r$, $1 < p<\infty$ and $ 1 \leq r < p$. Suppose $\Phi_{i}(t)= \int_0^t \phi_{i}(s)ds, \ i=1,2$,  are Young functions.
Then, with $p_1= \frac{p}{r}$, the following are equivalent:  

\begin{enumerate}
\item
 To each $T \in W((p,r),(\infty,\infty);\mu,\nu)$ there corresponds $C>0$ such that
\begin{equation}\label{eq: MT1 Kerman}
 \| Tf \|_{L_{\Phi_1}(Y, \nu)} \leq C \| f \|_{L_{\Phi_2}(X,\mu)},
\end{equation}
for all $f\in L_{\Phi_2}(X,\mu)$;
\item
There exist $B,D>0$ such that for all $x \in \mathbb{R_+}$,
\begin{equation}\label{Cianchi type condition Kerman}
\left( \int_x^{\infty} \frac{\phi_2(Dy)}{\Phi_2(Dy)^{p_1'}}y^{rp_1'} dy \right)^{\frac{1}{p_1'}}  \left( \int_0^{x}\frac{ \phi_1(y) }{y^{p}}dy \right)^{\frac{1}{p}} \leq B,
\end{equation}
namely,
\begin{equation*}\label{eq:cianchitype Kerman}
\left( \int_x^{\infty} \frac{\phi_2(Dy)}{\Phi_2(Dy)^{\frac{p}{p-r}}}y^{\frac{rp}{p-r}} dy \right)^{1-\frac{r}{p}}  \left( \int_0^{x}\frac{ \phi_1(y) }{y^{p}}dy \right)^{\frac{1}{p}} \leq B.
\end{equation*}
\end{enumerate}
\end{theorem}


\begin{proof}
We have shown in  \Cref{Calderon theorem}  and \Cref{Ineq equi Mod Hpr} that the inequality in (1) is equivalent to the weighted Hardy inequality (\ref{eq:HD}) for nonnegative, nonincreasing functions and hence to the inequality (\ref{eq:1rp main ineq}) for nonnegative, measurable functions. We need thus only show that (\ref{eq:CONDITIONS 1,2 modified by Kerman}) holds if and only if (\ref{Cianchi type condition Kerman}) does. It is important to observe that it follows from our previous considerations, (\ref{Cianchi type condition Kerman}) holds for $D'>D$ whenever it holds for $D$. 

Suppose, first, we have (\ref{eq:CONDITIONS 1,2 modified by Kerman}). Now, with $D=2^{\frac{1}{r}}C$
 
 \begin{equation*}
 \int_x^{\infty} \frac{\phi_2(Dy)}{\Phi_2(Dy)^{p_1'}}y^{rp_1'} dy   =  \left( \int_x^{2^{\frac{1}{r}}x} + \int_{2^{\frac{1}{r}}x}^{\infty}  \right) \frac{\phi_2(Dy)}{\Phi_2(Dy)^{p_1'}}y^{rp_1'} dy= I(x) +II(x).
\end{equation*}
Thus, the left side of (\ref{Cianchi type condition Kerman}) is dominated by
\begin{equation*}
I(x)^{\frac{1}{p_1'}}\left( \int_0^{x}\frac{ \phi_1(y) }{y^{p}}dy \right)^{\frac{1}{p_1}} + II(x)^{\frac{1}{p_1'}}\left( \int_0^{x}\frac{ \phi_1(y) }{y^{p}}dy \right)^{\frac{1}{p_1}}.
\end{equation*}
But, for $y \geq 2^{\frac{1}{r}}x$,
\begin{equation*}
\begin{split}
II(x) & \leq 2^{p_1'} \int_{2^{\frac{1}{r}}x}^{\infty} \left( 1-\frac{x^r}{y^r} \right)^{p_1'}\frac{\phi_2(Dy)}{\Phi_2(Dy)^{p_1'}} y^{rp_1'}dy\\
& \leq 2^{p_1'} \int_{x}^{\infty} \left( y^r-x^r \right)^{p_1'}\frac{\phi_2(Cy)}{\Phi_2(Cy)^{p_1'}} dy.
\end{split}
\end{equation*}
Therefore, the first condition in (\ref{eq:CONDITIONS 1,2 modified by Kerman}) ensures that
\begin{equation*}
II(x)^{\frac{1}{p_1'}}\left( \int_0^{x}\frac{ \phi_1(y) }{y^{p}}dy \right)^{\frac{1}{p_1}} \leq 2A.
\end{equation*}
Again, with $D=2^{\frac{1}{r}}C$,
\begin{equation*}
\begin{split}
I(x)^{\frac{1}{p_1'}}\left( \int_0^{x}\frac{ \phi_1(y) }{y^{p}}dy \right)^{\frac{1}{p_1}} & = \left( \int_x^{2^{\frac{1}{r}}x}    \frac{\phi_2(Dy)}{\Phi_2(Dy)^{p_1'}}y^{rp_1'} dy \right)^{\frac{1}{p_1'}}   \left( \int_0^{x}\frac{ \phi_1(y) }{y^{p}}dy \right)^{\frac{1}{p_1}}\\
& \leq 2x^{r}\left( \int_x^{2^{\frac{1}{r}}x}    \frac{\phi_2(Dy)}{\Phi_2(Dy)^{p_1'}} dy \right)^{\frac{1}{p_1'}}   \left( \int_0^{x}\frac{ \phi_1(y) }{y^{p}}dy \right)^{\frac{1}{p_1}}\\
& \leq 2^{1- \frac{1}{rp_1'}}\left( \int_{2^{\frac{1}{r}}x}^{4^{\frac{1}{r}}x}    \frac{\phi_2(Cy)}{\Phi_2(Cy)^{p_1'}} dy \right)^{\frac{1}{p_1'}}   \left(x^{p} \int_0^{x}\frac{ \phi_1(y) }{y^{p}}dy \right)^{\frac{1}{p_1}}.
\end{split}
\end{equation*}
Now,
\begin{equation*}
x^{p} \int_0^{x}\frac{ \phi_1(y) }{y^{p}}dy \leq  \int_0^{x} (2x^r-y^r)^{p_1} \frac{ \phi_1(y) }{y^{p}}dy \leq  \int_0^{2^{\frac{1}{r}}x} ((2^{\frac{1}{r}}x)^r-y^r)^{p_1} \frac{ \phi_1(y) }{y^{p}}dy.
\end{equation*}
Altogether, then,
\begin{equation*}
\begin{split}
I(x)^{\frac{1}{p_1'}}\left( \int_0^{x}\frac{ \phi_1(y) }{y^{p}}dy \right)^{\frac{1}{p_1}} 
& \leq 2^{1- \frac{1}{rp_1'}}\left( \int_{2^{\frac{1}{r}}x}^{\infty}    \frac{\phi_2(Cy)}{\Phi_2(Cy)^{p_1'}} dy \right)^{\frac{1}{p_1'}}   \left(\int_0^{2^{\frac{1}{r}}x} ((2^{\frac{1}{r}}x)^r-y^r)^{p_1} \frac{ \phi_1(y) }{y^{p}}dy \right)^{\frac{1}{p_1}}\\
& \leq 2^{1- \frac{1}{rp_1'}} A.
\end{split}
\end{equation*}

Next, assume (\ref{Cianchi type condition Kerman}) holds with constant $D>0$. Fix $x \in \mathbb{R_+}$. The left side of the first condition in (\ref{eq:CONDITIONS 1,2 modified by Kerman}) is equal to 
\begin{align*}
& \left(\int_{x}^{\infty} \left( 1-\frac{x^r}{y^r} \right)^{p_1'}\frac{\phi_2(Dy)}{\Phi_2(Dy)^{p_1'}} y^{rp_1'}dy \right)^{\frac{1}{p_1'}}   \left( \int_0^{x}\frac{ \phi_1(y) }{y^{p}}dy \right)^{\frac{1}{p_1}} \\
& \leq \left(\int_{x}^{\infty} \frac{\phi_2(Dy)}{\Phi_2(Dy)^{p_1'}} y^{rp_1'}dy \right)^{\frac{1}{p_1'}}   \left( \int_0^{x}\frac{ \phi_1(y) }{y^{p}}dy \right)^{\frac{1}{p_1}} \leq B.
\end{align*}
Again, 
\begin{align*}
& \left(\int_{x}^{\infty}   \frac{\phi_2(C_2y)}{\Phi_2(C_2y)^{p_1'}} dy \right)^{\frac{1}{p_1'}}
 \left( \int_0^{x} (x^r-y^r)^{p_1}  \frac{ \phi_1(y) }{y^{p}}dy \right)^{\frac{1}{p_1}}\\
& \leq \left(\int_{x}^{\infty} \frac{\phi_2(Dy)}{\Phi_2(Dy)^{p_1'}} x^{rp_1'}dy \right)^{\frac{1}{p_1'}}   \left( \int_0^{x}\frac{ \phi_1(y) }{y^{p}}dy \right)^{\frac{1}{p_1}}\\
& \leq \left(\int_{x}^{\infty} \frac{\phi_2(Dy)}{\Phi_2(Dy)^{p_1'}} y^{rp_1'}dy \right)^{\frac{1}{p_1'}}   \left( \int_0^{x}\frac{ \phi_1(y) }{y^{p}}dy \right)^{\frac{1}{p_1}} \leq B.
\end{align*}
\end{proof}

\subsection{\texorpdfstring {The case $ p \leq r <\infty$}{}}
We have shown in \Cref{Calderon type theorem} and \Cref{Second section} that $(L_{\Phi_2}(X,\mu), L_{\Phi_1}(Y,\nu))$ is an interpolation pair for $W((p,r),(\infty, \infty);\mu,\nu)$ if and only if there holds the modular inequality
\begin{equation*}
  \int_{\mathbb{R_+}}\Phi_1(H^{p,r}f^*(t))dt \leq \int_{\mathbb{R_+}}\Phi_2(Kf^*(s))ds,
\end{equation*}
for all $f \in M_+(\mathbb{R_+},m)$. We will prove that this modular inequality holds if and only if $\Phi_1$ and $\Phi_2$ satisfy the Zygmund-Str\"omberg condition: There exist $A,B>0$ such that for all $t \in \mathbb{R_+}$,
\begin{equation*}
t^p \int_0^t \frac{\Phi_1(s)}{s^{p+1}}ds \leq A \Phi_2(Bt).
\end{equation*}

Our complete result is
\begin{theorem}\label{main result 22 class Wpr}
Fix the indices $p$ and $r$, where $1 < p<\infty$ and $ p \leq r < \infty$. Let $(X,\mu)$ and $(Y,\nu)$ be $\sigma$-finite measure spaces with $\mu(X)=\nu(Y)=\infty$, the latter being nonatomic and separable.  Suppose $\Phi_1$ and $\Phi_2$ are Young functions.
Then, the following are equivalent:  

\begin{enumerate}
\item
 To each $T \in W((p,r),(\infty,\infty);\mu,\nu)$ there corresponds $C>0$ such that
\begin{equation}
 \| Tf \|_{L_{\Phi_1}(Y, \nu)} \leq C \| f \|_{L_{\Phi_2}(X,\mu)},
\end{equation}
whenever $f\in L_{\Phi_2}(X,\mu)$;
\item 
\begin{equation}\label{Modular inequality in main theorem r begger than p}
  \int_{\mathbb{R_+}}\Phi_1(H^{p,r}f^*(t))dt \leq \int_{\mathbb{R_+}}\Phi_2(Kf^*(s))ds,
\end{equation}
for all $f \in M_+(\mathbb{R_+},m)$
\item
One has 
\begin{equation}\label{eq:Tprcond}
t^p \int_0^t \frac{\Phi_1(s)}{s^{p+1}}ds \leq A \Phi_2(Bt),
\end{equation}
in which $A,B>0$ are independent of  $t \in \mathbb{R_+}$.
\end{enumerate}
\end{theorem}
\begin{proof}
Only the equivalence of (2) and (3) needs proving at this point.

The argument that (\ref{eq:Tprcond}) implies (\ref{Modular inequality in main theorem r begger than p}) is essentially that of Str\"omberg \cite{Str79} for the case $r=p$. Thus, by 
%
%
 \Cref{LDEHpr},
\begin{align*} 
\int_{\mathbb{R_+}}\Phi_1((H^{p,r}f^*)(t))dt
& \leq \int_{\mathbb{R_+}}\Phi_1((H^{p,p}f^*)(t))dt \\
& = \int_{\mathbb{R_+}} \phi_1(t)m_{H^{p,p}f^*}(t)dt \\
& \leq 2^{2p+1}p \int_{\mathbb{R_+}}    \frac{\phi_1(t)}{t^p} \int_{t/8}^{\infty} m_{f^*}(s)s^{p-1}ds dt \\
& = 2^{2p+1}p \int_{\mathbb{R_+}} m_{f^*}(s)s^{p-1}    \int_{0}^{8s} \frac{\phi_1(t)}{t^p}  dt \, ds  \\
& = \frac{p}{2^{p-1}} \int_{\mathbb{R_+}} m_{f^*}(s)  {(8s)}^{p}    \int_{0}^{8s} \frac{\Phi_1(t)}{t^{p+1}}  dt \, \frac{ds}{s}  \\
& \leq \frac{p}{2^{p-1}}A \int_{\mathbb{R_+}} m_{f^*}(s)  \Phi_2(8Bs) \frac{ds}{s}  \\
& \leq \frac{p}{2^{p-1}}A \int_{\mathbb{R_+}} m_{f^*}(s)  \phi_2(8Bs) d(8Bs)  \\
& = \frac{p}{2^{p-1}}A \int_{\mathbb{R_+}} m_{8Bf^*}(s)  \phi_2(s) ds  \\
& = \frac{p}{2^{p-1}}A \int_{\mathbb{R_+}}\Phi_2(8Bf^*(s))ds\\
& \leq  \int_{\mathbb{R_+}}\Phi_2(Kf^*(s))ds.
\end{align*}
where $K= \frac{p}{2^{p-1}}8AB$ or $8B$, according as $\frac{p}{2^{p-1}}A \geq 1$ or not.

To obtain (\ref{eq:Tprcond}) from (\ref{Modular inequality in main theorem r begger than p}) we substitute $f(s)=f^*(s)=t \chi_{(0,1)}(s)$ in (\ref{Modular inequality in main theorem r begger than p}) to get
\begin{align*}
\int_{\mathbb{R_+}}\Phi_2(Kf^*(u))du 
& \geq \int_1^{\infty} \Phi_1(H^{p,r}f^*(u))du \\
&= \int_1^{\infty} \Phi_1 \left( u^{-\frac{1}{p}}  \left( \int_0^u f^*(s)^r s^{\frac{r}{p}-1}ds \right)^{\frac{1}{r}}  \right)du\\
&= \int_1^{\infty} \Phi_1 \left( u^{-\frac{1}{p}} \left( {\textstyle\frac{p}{r}} \right)^{\frac{1}{r}}t \right)du \\
&= \int_1^{\infty} \Phi_1 \left( u^{-\frac{1}{p}} \gamma t \right)du\\
&= p (\gamma t)^p \int_0^{\gamma t} \frac{\Phi_1(s)}{s^{p+1}}ds,
\end{align*}
where $\gamma = \left( \frac{p}{r} \right)^{\frac{1}{r}}$. Since $\int_{\mathbb{R_+}}\Phi_2(Kf^*(u))du = \int_0^1 \Phi_2(Kt)du = \Phi_2(Kt)$, we find, on replacing $t$ by $\frac{t}{\gamma}$, that (\ref{eq:Tprcond}) is satisfied with $A=\frac{1}{p}$ and $B=\frac{K}{\gamma}=K \left( \frac{r}{p} \right)^{\frac{1}{r}}$.
\end{proof}

\section{Interpolation results for the class $W((1,1),(q,r);\mu,\nu)$}\label{Class W(11qr)}
Recall that a  $r$-quasilinear operator $T$ is in  $W((1,1),(q,r);\mu,\nu)$ if
\begin{equation*}
T : L_{1}(X,\mu) \rightarrow L_{1,\infty}(Y,\nu)  \ \ \  \text{and} \ \ \ T : L_{q,r}(X,\mu) \rightarrow L_{q,\infty}(Y,\nu).
\end{equation*}

Our main results in this section are \Cref{main result for the Wqr class 1} and \Cref{main result for the Wqr class 2}  which give necessary and sufficient conditions on Young functions $\Phi_1$ and $\Phi_2$ so that every $T \in W((1,1),(q,r);\mu,\nu)$ maps $L_{\Phi_2}(X,\mu)$ boundedly into $L_{\Phi_1}(Y,\nu)$; the first theorem deals with $1\leq r <q$, the second with $q \leq r < \infty$.

We proceed as in \Cref{Class W(prInfinity)}. First, in \Cref{Calderon theorem 11qr}, we reduce the problem to the boundedness of a Calder\'on operator, $S_{q,r}$, which corresponds to the class $W((1,1),(q,r);\mu,\nu)$. We then establish, in \Cref{relation boundedness of Sqr with weighted Hardy ineq},  the equivalence between a gauge norm inequality involving $S_{q,r}$ and a certain weighted Hardy inequality. The desired characterizations are then obtained in much the same way as those in the previous section.
\subsection{A Calder\'on-type theorem}\label{Calderon theorem 11qr}
We use the notation $S_{q,r}$ for the Calder\'on operator $P+H_{q,r}$, $(1<q< \infty, 1 \leq r < \infty)$ where, for $g \in M_+(\mathbb{R_+},m), t \in \mathbb{R_+}$,
\begin{align*}
(Pg)(t)=  \frac{1}{t} \int_0^t g(s)ds \ \ \ \text{and} \ \ \ (H_{q,r}g)(t)=  \left( t^{-\frac{r}{q}} \int_{t}^{\infty}  g(s)^r s^{\frac{r}{q}-1}ds \right)^{\frac{1}{r}}.
\end{align*}
The operator $P$ is same as $H^{1,1}$, the Hardy averaging operator, but we prefer to use the more familiar notation $P$ for it.

We begin with the following analogue of \Cref{Hprdominance}.
\begin{theorem}\label{Sqrdominance}
Let $(X,\mu)$ and $(Y,\nu)$ be $\sigma$-finite measure spaces and fix the indices $q$ and $r$, $1<q<\infty$ and $ 1 \leq r < \infty $. Suppose $T$ is an $r$-quasilinear operator in the class $W((1,1),(q,r);\mu,\nu)$. Then,
\begin{equation}\label{eq:Hqr}
(Tf)^*(t) \leq K \, \left( S_{q,r}f^* \right)(ct),
\end{equation}
where $K$ is independent of $f \in (L_{1}+L_{q,r})(X,\mu)$ and $t \in \mathbb{R_+}$.

Further, the operator $S_{q,r}$ is in the class $W((1,1),(q,r);m,m)$.
\end{theorem}

\begin{remark}
The argument in proving the above theorem is the same as that of \Cref{Hprdominance}, so we have skipped its proof. We point out that, here, $K= \max \left[ M_{q,r}  \left( {\textstyle\frac{r}{q}} \right)^{\frac{1}{r}}   {\textstyle\frac{1  }{ (1- c^{\frac{r}{q}})^{\frac{1}{r}}}},4M_1 \right]C$, where $c$ and $C$ are the constant of $r$-quasilinearity of $T$ and $M_{1}$, $M_{q,r}$ are operator norms of the mappings $T: L_{1}(X,\mu) \rightarrow L_{1,\infty}(Y,\nu)$  and $T: L_{q,r}(X,\mu) \rightarrow L_{q,\infty}(Y,\nu)$ respectively
\end{remark}

Now we establish a Calder\'on-type interpolation theorem for operators in $W((1,1),(q,r);$ $\mu,\nu)$.

\begin{theorem}\label{Calderon type theorem for Sqr}
Fix the indices $q$ and $r$, where $1 < q < \infty$ and $1 \leq r < \infty$. Suppose $(X,\mu)$ and $(Y,\nu)$ are $\sigma$-finite measure spaces with $\mu(X)=\nu(Y)=\infty$, the latter being nonatomic and separable. Then, the  following are equivalent:
\begin{enumerate}
\item[(1)] Every operator $T$ in the class $W((1,1),(q,r);\mu,\nu)$ is bounded from $L_{\Phi_2}(X,\mu)$ to $L_{\Phi_1}(Y,\nu)$;
\item[(2)] The operator $S_{q,r}$ is bounded from $L_{\Phi_2}(\mathbb{R_+},m)$ to $L_{\Phi_1}(\mathbb{R_+},m)$.
\end{enumerate}
\end{theorem}
\begin{proof}
We only show $(1)$ implies $(2)$, since the converse can be settled as in \Cref{Calderon theorem}.

In \Cref{construction of Pull back of operator S} take $S=S_{q,r}$ and denote by $\tilde{S}_{q,r}$ the operator $\tilde{S}$ guaranteed to exist by that theorem, so that
\begin{equation*}
(\tilde{S}_{q,r}f)^{*_{\nu}}=  S_{q,r}f^{*_{\mu}}, \ \ \ m \text{-a.e.},
\end{equation*}
for all $f \in M(X,\mu)$, since $\left( S_{q,r}f^{*_{\mu}} \right)(t)= \int_0^1 f^{*_{\mu}}(ts)ds + \left( \int_1^{\infty} f^{*_{\mu}}(ts)^{r} s^{\frac{r}{q}-1}ds \right)^{\frac{1}{r}}$, is nonincreasing, so $\left( S_{q,r}f^{*_{\mu}} \right)^{*}= S_{q,r}f^{*_{\mu}}$. Moreover, \Cref{S-tilda is nu-quasilinear if S is} ensures that $\tilde{S}_{q,r}$ is $r$-quasilinear, since $S_{q,r}$ is.

Again arguing as in the proof of \Cref{Calderon theorem}, one obtains
\begin{equation*}
\tilde{S}_{q,r} : L_{1}(X,\mu) \rightarrow L_{1,\infty}(Y,\nu) \ \ \ \text{and} \ \ \ \tilde{S}_{q,r} : L_{q,r}(X,\mu) \rightarrow L_{q,\infty}(Y,\nu)
\end{equation*}
 boundedly, hence
 \begin{equation*}
 \tilde{S}_{q,r} : L_{\Phi_2}(X,\mu) \rightarrow L_{\Phi_1}(Y,\nu)
 \end{equation*}
boundedly and so 
\begin{equation*}
S_{q,r} : L_{\Phi_2}(\mathbb{R_+},m) \rightarrow L_{\Phi_1}(\mathbb{R_+},m),
\end{equation*}
thereby completing the proof.
\end{proof}

\subsection{\texorpdfstring{The Calder\'on  operator $S_{q,r}$ and an associated Hardy inequality}{}}\label{relation boundedness of Sqr with weighted Hardy ineq}
In this section, we give the connection between the norm inequality of the operator $S_{q,r}$ and a weighted Hardy inequality.

An easy exercise in changes of variable shows that $S_{q,r}$ is a dilation-commuting operator, so from \cite[Theorem A]{KRS17}, we get the following result.
\begin{theorem}\label{EquvlcHqrnormandmod}
Let $\Phi_1$ and $\Phi_2$ be Young functions, and fix indices $q$ and $r$, with $1<q< \infty $, $1 \leq r< \infty$. Then,  norm inequality
\begin{equation}\label{eq:normmmS}
\parallel S_{q,r}f \parallel _{L_{\Phi_1}(\mathbb{R_+}, m)} \leq C \parallel f \parallel_{L_{\Phi_2}(\mathbb{R_+}, m)}
\end{equation}
holds for all $f \in  M_+(\mathbb{R_+},m)$ if and only if the modular inequality
\begin{equation}\label{eq:modddS}
  \int_{\mathbb{R_+}}\Phi_1(S_{q,r}(f^*(t)))dt \leq \int_{\mathbb{R_+}}\Phi_2(Kf^*(s))ds ,
\end{equation}
holds for all $f$ in $M_+(\mathbb{R_+},m)$.
\end{theorem}

In the next theorem we estimate the distribution function of $Tf$, where $ T \in W((1,1),$  $(q,r);\mu,\nu)$ and $f \in (L_{1} + L_{q,r})(X,\mu)$.
\begin{lemma}\label{distributional estimate T in class W(11qr)}
Let $1<q< \infty,1 \leq r < \infty$ and suppose $T \in W((1,1),(q,r);\mu,\nu)$. Then, for every $f \in (L_{1} + L_{q,r})(X,\mu)$, $t \in \mathbb{R_+}$ and for every $k>0$,
\begin{multline}\label{upper estimate of nuTf involving k}
\nu_{Tf}(t) \leq c^{-1} \frac{4CM_1}{t} \left(  \int_{\frac{t}{4Ck}}^{\infty}\mu_f(s)ds \right) 
 + c^{-1} \left( \textstyle\frac{4C r^{1/r} M_{q,r}}{t} \right)^q   \left( \int_0^{\frac{t}{4Ck}} \mu_{f}(s')^{r/q}s'^{r-1}ds' \right)^{q/r},
\end{multline}
where $c$ and $C$ are the constant of $r$-quasilinearity of $T$ and $M_{1}$, $M_{q,r}$ are operator norms of the mappings $T: L_{1}(X,\mu) \rightarrow L_{1,\infty}(Y,\nu)$  and $T: L_{q,r}(X,\mu) \rightarrow L_{q,\infty}(Y,\nu)$ respectively.

In particular, for $k=M_1$,
\begin{multline}\label{eq:DEHq}
\nu_{Tf}(t) \leq c^{-1} {\left( \textstyle\frac{t}{4CM_1} \right)^{-1}}    \int_{\frac{t}{4CM_1}}^{\infty}\mu_f(s)ds \\
 + c^{-1} \left( \textstyle\frac{r^{1/r} M_{q,r}}{M_1} \right)^q     \left( {\textstyle\frac{t}{4CM_1}} \right)^{-q}  \left( \int_0^{\frac{t}{4CM_1}} \mu_{f}(s)^{r/q}s^{r-1}ds \right)^{q/r}.
\end{multline}
\end{lemma}
\begin{proof}
Fix $t>0$ and $f \in (L_{1} + L_{q,r})(X,\mu)$. Let $k$ be any positive number. Write $f=f_t+f^t$, with
\begin{equation*}
    f^t(x)= \begin{cases}
                 f(x) & |f(x)|> {\textstyle\frac{t}{2Ck}},\\
                 0  & |f(x)| \leq {\textstyle\frac{t}{2Ck}},
                \end{cases}
\end{equation*}
and $f_t(x)= f(x)- f^t(x)$. Then $f^t \in L_1(X,\mu)$ and $f_t \in L_{q,r}(X,\mu)$.

Now, by the $r$-quasilinearity  of $T$,
\begin{equation*}
\nu_{Tf}(t) \leq c^{-1} \left[ \nu_{Tf^t}\left({\textstyle\frac{t}{2C}} \right)+\nu_{Tf_t}\left( {\textstyle\frac{t}{2C}} \right) \right].
\end{equation*}

Since, $T : L_{1}(X,\mu) \rightarrow L_{1,\infty}(Y,\nu)$ with operator norm $M_{1}$, we have, 
\begin{equation*}
\begin{split}
\sup\limits_{y>0} y \, \nu_{Tf^t}(y) & \leq M_{1} \|f^t\|_{L_{1}}\\
& = M_1 \left( \int_{0}^{\frac{t}{2Ck}} \mu_f \left( \textstyle\frac{t}{2Ck}  \right) ds  + \int_{\frac{t}{2Ck}}^{\infty}\mu_f(s)ds \right).
\end{split}
\end{equation*}
It follows that
\begin{equation}\label{eq:DETf unbddpart}
\begin{split}
\nu_{Tf^t}\left( \textstyle\frac{t}{2C} \right) & \leq  {\textstyle\frac{M_1}{  \left( {\textstyle\frac{t}{2C}} \right) }}  \left( { \textstyle\frac{t}{2Ck}  \mu_f \left( \textstyle\frac{t}{2Ck}  \right)}   + \int_{\frac{t}{2Ck}}^{\infty}\mu_f(s)ds \right)\\
& \leq  {\textstyle\frac{M_1}{  \left( {\textstyle\frac{t}{2C}} \right) }}  \left(  2\int_{\frac{t}{4Ck}}^{\infty}\mu_f(s)ds \right)\\
& = \frac{4CM_1}{t} \left(  \int_{\frac{t}{4Ck}}^{\infty}\mu_f(s)ds \right) ,
\end{split}
\end{equation}
where the last but one inequality follows from the fact that, for any $x \in \mathbb{R_+}$,
\begin{equation*}
\begin{split}
\int_x^{\infty} \mu_f(s)ds & = \int_x^{2x} \mu_f(s)ds + \int_{2x}^{\infty} \mu_f(s)ds \\
& \geq   \mu_f(2x)x + \int_{2x}^{\infty} \mu_f(s)ds \\
& \geq {\textstyle\frac{1}{2}} \left[  \mu_f(2x)2x + \int_{2x}^{\infty} \mu_f(s)ds  \right],
\end{split}
\end{equation*}
which yields the assertion on taking $x=t/4Ck$.

Again, since $T : L_{q,r}(X,\mu) \rightarrow L_{q,\infty}(Y,\nu)$, with operator norm $M_{q,r}$, we have,
\begin{equation*}
\begin{split}
\sup\limits_{y>0} y \, \nu_{Tf_t}(y)^{\frac{1}{q}} & \leq M_{q,r} \|f_t\|_{L_{q,r}(X,\mu)}\\
& =  M_{q,r}   \left(  r \int_{\mathbb{R_+}} \mu_{f_t}(s)^{r/q}s^{r-1}ds \right)^{1/r}\\
& = r^{1/r} M_{q,r}   \left( \int_0^{\frac{t}{2Ck}} \left( \mu_{f}(s) - \mu_{f} \left( \textstyle\frac{t}{2Ck} \right)  \right)^{r/q}s^{r-1}ds \right)^{1/r}\\
& \leq r^{1/r} M_{q,r}   \left( \int_0^{\frac{t}{2Ck}} \mu_{f}(s)^{r/q}s^{r-1}ds \right)^{1/r},
\end{split}
\end{equation*}
whence,
\begin{equation}\label{eq:DETf bddpart}
\begin{split}
\nu_{Tf_t}\left( \textstyle\frac{t}{2C} \right) 
& \leq \left( \textstyle\frac{2Cr^{1/r} M_{q,r}}{t} \right)^q   \left( \int_0^{\frac{t}{2Ck}} \mu_{f}(s)^{r/q}s^{r-1}ds \right)^{q/r}\\
& = \left( \textstyle\frac{4C r^{1/r} M_{q,r}}{t} \right)^q   \left( \int_0^{\frac{t}{4Ck}} \mu_{f}(2s')^{r/q}s'^{r-1}ds' \right)^{q/r}\\
& \leq  \left( \textstyle\frac{4C r^{1/r} M_{q,r}}{t} \right)^q   \left( \int_0^{\frac{t}{4Ck}} \mu_{f}(s')^{r/q}s'^{r-1}ds' \right)^{q/r}.
\end{split}
\end{equation}
From the estimates (\ref{eq:DETf unbddpart}) and (\ref{eq:DETf bddpart}), we have
\begin{multline}
\nu_{Tf}(t) \leq c^{-1} \frac{4CM_1}{t} \left(  \int_{\frac{t}{4Ck}}^{\infty}\mu_f(s)ds \right) 
 + c^{-1} \left( \textstyle\frac{4C r^{1/r} M_{q,r}}{t} \right)^q   \left( \int_0^{\frac{t}{4Ck}} \mu_{f}(s')^{r/q}s'^{r-1}ds' \right)^{q/r}.
\end{multline}

Whence, setting $k=M_1$, we get the following pointwise estimate for $\nu_{Tf}$,
\begin{multline*}
\nu_{Tf}(t) \leq \frac{c^{-1}}{ t/4CM_1}    \int_{\frac{t}{4CM_1}}^{\infty}\mu_f(s)ds 
 + c^{-1} \left( \textstyle\frac{r^{1/r} M_{q,r}}{M_1} \right)^q   \frac{1}{  \left( t/4CM_1 \right)^q}  \left( \int_0^{\frac{t}{4CM_1}} \mu_{f}(s)^{r/q}s^{r-1}ds \right)^{q/r}.
\end{multline*}
\end{proof}


\begin{remark}\label{remark for upper bound of distribution of Sqrf}
In \Cref{distributional estimate T in class W(11qr)} take $X=Y= \mathbb{R_+}$, $\mu=\nu=m$ and
\begin{equation*}
(Tf)(t)= (S_{q,r}f^*)(t), \ \ \ f \in (L_1+L_{q,r})(\mathbb{R_+},m).
\end{equation*}
For this operator we have $C=2^{1-\frac{1}{r}}$, $c= {\textstyle\frac{1}{2}}$, $M_{1} \leq  1+\left( \textstyle\frac{q'}{r} \right)^{\frac{1}{r}}$ and 
\begin{equation*}
M_{q,r} \leq \begin{cases}
              \left[  \left( \frac{q'}{r'} \right)^{\frac{1}{r'}}+1 \right]  \left( \frac{q}{r} \right)^{\frac{1}{r}},  & r>1, \\
               2q,               & r=1;
       \end{cases}
\end{equation*}
indeed,
\begin{align*}
t^{\frac{1}{q}} (Pf^*)(t) & = t^{\frac{1}{q}} t^{-1} \int_0^t f^*(s)ds \\
& = t^{-\frac{1}{q'}}  \int_0^t \left[  f^*(s) s^{\frac{1}{q}  -\frac{1}{r} } \right]  \ \ s^{\frac{1}{r} - \frac{1}{q} } ds \\
& \leq t^{-\frac{1}{q'}}  \left(  \int_0^t  f^*(s)^r s^{\frac{r}{q}  -1 } ds \right)^{\frac{1}{r}}   \left( \int_0^t s^{\left( \frac{1}{r} - \frac{1}{q} \right)r' } ds \right)^{\frac{1}{r'}}\\
& =\left( \frac{q'}{r'} \right)^{\frac{1}{r'}}   t^{-\frac{1}{q'}}  \left(  \int_0^t  f^*(s)^r s^{\frac{r}{q}  -1 } ds \right)^{\frac{1}{r}} t^{\frac{1}{q'}} \\
& =\left( \frac{q'}{r'} \right)^{\frac{1}{r'}}  \left( \frac{q}{r} \right)^{\frac{1}{r}}   \left( {\textstyle\frac{r}{q}} \int_0^{t}  f^*(s)^r s^{\frac{r}{q}  -1 } ds \right)^{\frac{1}{r}}  \\
&  \leq  \left( \frac{q'}{r'} \right)^{\frac{1}{r'}}  \left( \frac{q}{r} \right)^{\frac{1}{r}}  \| f^* \|_{L_{q,r}(\mathbb{R_+},m)},
\end{align*}
and
\begin{align*}
t^{\frac{1}{q}} (H_{q,r}f^*)(t) & =    \left(  \int_t^{\infty}  f^*(s)^r s^{\frac{r}{q}  -1 } ds \right)^{\frac{1}{r}}  \\
& =  \left( \frac{q}{r} \right)^{\frac{1}{r}}   \left( {\textstyle\frac{r}{q}} \int_0^{t}  f^*(s)^r s^{\frac{r}{q}  -1 }  ds\right)^{\frac{1}{r}}  \\
&  \leq   \left( \frac{q}{r} \right)^{\frac{1}{r}}  \| f^* \|_{L_{q,r}(\mathbb{R_+},m)}.
\end{align*}
Set $\beta = 2^{3-\frac{1}{r}} \left( 1+ \left( \frac{q'}{r} \right)^{\frac{1}{r}} \right)$. In (\ref{upper estimate of nuTf involving k}), choose $k$ such that $4Ck= \beta$ and $\alpha$ such that 
\begin{equation*}
 \alpha = \begin{cases}    
                     q^{\frac{q}{r}}  \left( \frac{ \left( \frac{q'}                                                                                                                                   {r'} \right)^{\frac{1}{r'}} +1         }{    \left( \frac{q'}{r} \right)^{\frac{1}{r}} +1     } \right)^q,                              & r>1\\
                     q^{q} \left( \frac{2}{q'+1} \right)^{q}, & r=1,
                     \end{cases}
\end{equation*}
Then, such $\alpha$ satisfy $2 \alpha \beta^{q} \geq c^{-1}   \left( 4C r^{1/r} M_{q,r} \right)^q$ and we arrive at the estimate
\begin{align}
 m_{S_{q,r}f^*}(t) \leq 2\left[  \frac{1}{ t/ \beta }    \int_{\frac{t}{ \beta}}^{\infty} m_{f^*}(s)ds 
 +   \frac{ \alpha}{  \left( t/  \beta \right)^q}  \left( \int_0^{  \frac{t}{\beta }   } m_{f^*}(s)^{r/q}s^{r-1}ds \right)^{q/r} \right].
\end{align}
\end{remark}

\begin{theorem}\label{LDE for generalized Calderon operator}
Fix the indices $q$ and $r$, with $1  <q < \infty$ and $1 \leq r < \infty$. Then, for $f \in (L_{1}+L_{q,r})(\mathbb{R_+},m)$ and $t \in \mathbb{R_+}$, one has
we have,
 \begin{equation}
\begin{aligned}\label{lower estimate Calderon type operator}
{\textstyle\frac{1}{2^q   \left( 1+ \left( \frac{r}{q} \right)^{\frac{q}{r}} \right)  } } & \left[ t^{-1}      \int_{t}^{\infty} m_{f^*}(\lambda) d \lambda     +  t^{-q} \left(  \int_0^{t} m_{f^*}(\lambda)^{r/q} \lambda^{r-1} d \lambda  \right)^{q/r}   \right] \\
& \leq m_{S_{q,r}f^*}(t) \\
 & ~~~~~~~~~~~ \leq  
 E \left[  \frac{1}{ t/ \beta }    \int_{\frac{t}{\beta}}^{\infty} m_{f^*}(s)ds 
 +   \frac{ 1}{  \left( t/  \beta \right)^q}  \left( \int_0^{t/  \beta} m_{f^*}(s)^{r/q}s^{r-1}ds \right)^{q/r} \right].
 \end{aligned}
\end{equation}
in which $E= 2 \max \left[  1,  q^{\frac{q}{r}}  \left( \frac{      \left( \frac{q'}{r'} \right)^{\frac{1}{r'}} +1         }{    \left( \frac{q'}{r} \right)^{\frac{1}{r}} +1     } \right)^q \right]  $ and $\beta = 2^{3-\frac{1}{r}} \left( 1+ \left( \frac{q'}{r} \right)^{\frac{1}{r}} \right)$.
\end{theorem}
\begin{proof}
Fix $f \in (L_{1}+L_{q,r})(\mathbb{R_+},m)$ and $t \in \mathbb{R_+}$. To the end of establishing the first estimate in (\ref{lower estimate Calderon type operator}), let $\tau_0$ be the least $\tau$ such that $S_{q,r}f^*(\tau)=t$. Then,
\begin{equation*}
m_{S_{q,r}f^*}(t)= \tau_0 \ \ \  \text{and} \ \ \ (S_{q,r}f^*)(\tau_0)=t.
\end{equation*}
Observe that
\begin{equation*}
\int_0^{\tau_0} f^*(s)ds= \tau_0 f^*(\tau_0)+ \int_{f^*(\tau_0)}^{\infty} m_{f^*}(\lambda)d\lambda.
\end{equation*}
Now,
\begin{equation*}
\int_{\tau}^{\infty}f^*(s)^r s^{\frac{r}{q}-1}ds = {\textstyle\frac{q}{r}} \int_{\tau^{\frac{r}{q}}}^{\infty}f^*(u^{\frac{q}{r}})^r du= {\textstyle\frac{q}{r}} \int_0^{\infty}g(u)du={\textstyle\frac{q}{r}} \int_0^{\infty} m_{g}(\lambda) d \lambda
\end{equation*}
where
\begin{equation*}
g(u) = \begin{cases}
               0,  & u < \tau^{\frac{r}{q}}, \\
               f^*(u^{\frac{q}{r}})^r,               & u \geq \tau^{\frac{r}{q}}
       \end{cases}
\end{equation*}
and
\begin{equation*}
m_{g}(\lambda) = \begin{cases}
               0,  & \lambda \geq f^*(\tau)^{r}, \\
               m_{f*}(\lambda^{\frac{1}{r}})^{\frac{r}{q}}-\tau^{\frac{r}{q}},               & \lambda < f^*(\tau)^{r},
       \end{cases}
\end{equation*}
so,
\begin{equation*}
\int_{\tau}^{\infty}f^*(s)^r s^{\frac{r}{q}-1}ds= {\textstyle\frac{q}{r}} \left[   \int_0^{f^*(\tau)^r} m_{f^*}(\lambda^{\frac{1}{r}})^{\frac{r}{q}} d \lambda - \tau^{\frac{r}{q}} f^*(\tau)^r \right].
\end{equation*}
Thus,
\begin{align*}
\begin{split}
t = (S_{q,r}f^*)(\tau_0) &= f^*(\tau_0)+ \tau_0^{-1}\int_{f^*(\tau_0)}^{\infty} m_{f^*}(\lambda)d\lambda \\
& ~~~~~~~~~~~~+\tau_0^{-\frac{1}{q}} {\left( \textstyle\frac{q}{r} \right)}^{\frac{1}{r}} \left(  \int_0^{f^*(\tau_0)^r} m_{f^*}(\lambda^{\frac{1}{r}})^{\frac{r}{q}} d \lambda - \tau_0^{\frac{r}{q}} f^*(\tau_0)^r \right)^{\frac{1}{r}}\\
& \geq \min \left( 1, {\left( \textstyle\frac{q}{r} \right)}^{\frac{1}{r}} \right)     \left[    \tau_0^{-1}\int_{f^*(\tau_0)}^{\infty} m_{f^*}(\lambda)d\lambda + \tau_0^{-\frac{1}{q}} \left(  \int_0^{f^*(\tau_0)^r} m_{f^*}(\lambda^{\frac{1}{r}})^{\frac{r}{q}} d \lambda  \right)^{\frac{1}{r}} \right].
\end{split}
\end{align*}
Since
\begin{equation*}
f^*(\tau_0) \leq \tau_0^{-1}\int_0^{\tau_0} m_{f^*}(s)ds \leq (S_{q,r}f^*)(\tau_0)=t,
\end{equation*}
we have
\begin{equation*}
\begin{split}
\int_0^{f^*(\tau_0)^r} m_{f^*}(\lambda^{\frac{1}{r}})^{\frac{r}{q}} d \lambda 
& = \int_0^{t^r} m_{f^*}(\lambda^{\frac{1}{r}})^{\frac{r}{q}} d \lambda - \int^{t^r}_{f^*(\tau_0)^r} m_{f^*}(\lambda^{\frac{1}{r}})^{\frac{r}{q}} d \lambda\\
& \geq \int_0^{t^r} m_{f^*}(\lambda^{\frac{1}{r}})^{\frac{r}{q}} d \lambda - \tau_0^{\frac{r}{q}} \left( t^r -f^*(\tau_0)^r \right).
\end{split}
\end{equation*}
Altogether, then,
\begin{equation*}
\begin{split}
\max \left( 1, {\left( \textstyle\frac{r}{q} \right)}^{\frac{1}{r}} \right)  t 
& =  \max \left( 1, {\left( \textstyle\frac{r}{q} \right)}^{\frac{1}{r}} \right)    (S_{q,r}f^*)(\tau_0)\\
& \geq \tau_0^{-1}\int_{f^*(\tau_0)}^{\infty} m_{f^*}(\lambda)d\lambda + \tau_0^{-\frac{1}{q}} \left(\int_0^{t^r} m_{f^*}(\lambda^{\frac{1}{r}})^{\frac{r}{q}} d \lambda - \tau_0^{\frac{r}{q}} \left( t^r -f^*(\tau_0)^r \right) \right)^{\frac{1}{r}}\\
& \geq \tau_0^{-1}\int_{t}^{\infty} m_{f^*}(\lambda)d\lambda +   \tau_0^{-\frac{1}{q}} \left( \int_0^{t^r} m_{f^*}(\lambda^{\frac{1}{r}})^{\frac{r}{q}} d \lambda\right)^{\frac{1}{r}} -t
\end{split}
\end{equation*}
or
\begin{equation*}
\max \left( 2, {1+\left( \textstyle\frac{r}{q} \right)}^{\frac{1}{r}} \right)  t \geq \tau_0^{-1}\int_{t}^{\infty} m_{f^*}(\lambda)d\lambda +   \tau_0^{-\frac{1}{q}} \left( \int_0^{t^r} m_{f^*}(\lambda^{\frac{1}{r}})^{\frac{r}{q}} d \lambda\right)^{\frac{1}{r}}.
\end{equation*}
From
\begin{equation*}
\max \left( 2, {1+\left( \textstyle\frac{r}{q} \right)}^{\frac{1}{r}} \right) t \geq \tau_0^{-1}\int_{t}^{\infty} m_{f^*}(\lambda)d\lambda
\end{equation*}
and
\begin{equation*}
\max \left( 2, {1+\left( \textstyle\frac{r}{q} \right)}^{\frac{1}{r}} \right)   t \geq \tau_0^{-\frac{1}{q}} \left( \int_0^{t^r} m_{f^*}(\lambda^{\frac{1}{r}})^{\frac{r}{q}} d \lambda\right)^{\frac{1}{r}},
\end{equation*}
we deduce that, with $\gamma = \max \left( 2, {1+\left( \textstyle\frac{r}{q} \right)}^{\frac{1}{r}} \right)$,
\begin{equation*}
\tau_0  \geq \frac{1}{\gamma} \, t^{-1} \int_{t}^{\infty} m_{f^*}(\lambda)d\lambda \geq \frac{1}{\gamma^q} \, t^{-1} \int_{t}^{\infty} m_{f^*}(\lambda)d\lambda
\end{equation*}
and
\begin{equation*}
\tau_0  \geq  \frac{1}{\gamma^q} \, t^{-q} \left( \int_0^{t^r} m_{f^*}(\lambda^{\frac{1}{r}})^{\frac{r}{q}} d \lambda\right)^{\frac{q}{r}},
\end{equation*}
so
\begin{equation*}
\tau_0 \geq \frac{1}{2\gamma^q}   \left[ t^{-1}\int_{t}^{\infty} m_{f^*}(\lambda)d\lambda + t^{-q} \left( \int_0^{t^r} m_{f^*}(\lambda^{\frac{1}{r}})^{\frac{r}{q}} d \lambda\right)^{\frac{q}{r}} \right].
\end{equation*}

Now,
\begin{align*}
2 \gamma^q & = 2 \max \left[2^q, \left( {1+\left( \textstyle\frac{r}{q} \right)}^{\frac{1}{r}} \right)^q \right]\\
& \leq 2 \max \left[2^q, 2^{q-1} \left( {1+\left( \textstyle\frac{r}{q} \right)}^{\frac{q}{r}} \right) \right]\\
& = \max \left[2^{q+1}, 2^{q} \left( {1+\left( \textstyle\frac{r}{q} \right)}^{\frac{q}{r}} \right) \right].
\end{align*}
Therefore,
\begin{equation*}
\tau_0 \geq \frac{1}{2^q} \min \left[ \frac{1}{2},  \left( {1+\left( \textstyle\frac{r}{q} \right)}^{\frac{q}{r}} \right)^{-1} \right]  \left[ t^{-1}\int_{t}^{\infty} m_{f^*}(\lambda)d\lambda + t^{-q} \left( \int_0^{t^r} m_{f^*}(\lambda^{\frac{1}{r}})^{\frac{r}{q}} d \lambda\right)^{\frac{q}{r}} \right].
\end{equation*}

This, together with the upper bound obtained for $(S_{q,r}f^*)(t)$ in \Cref{remark for upper bound of distribution of Sqrf}, completes the proof.
\end{proof}

\begin{theorem}\label{Ineq equi Mod Hqr}
Let $(X,\mu)$ and $(Y,\nu)$ be $\sigma$-finite  measure spaces, with $(Y,\nu)$  being nonatomic and separable. Fix indices $q$ and $r$, where $1 < q<\infty$ and $ 1 \leq r < \infty$. Suppose $\Phi_1$ and $\Phi_2$ are Young functions, with $\phi_{i}(t)= \textstyle\frac{d\Phi_{i}}{dt}(t)$, $i=1,2$. Then, the following are equivalent:  

\begin{enumerate}
\item
 To each $T \in  W((1,1),(q,r);\mu,\nu)$ there corresponds a $C>0$ such that
\begin{equation}\label{eq: MT2}
 \| Tf \|_{L_{\Phi_1}(Y, \nu)} \leq C \| f \|_{L_{\Phi_2}(X,\mu)},
\end{equation}
for all $f\in L_{\Phi_2}(X,\mu)$;
\item
There exist constants $C_1,C_2>0$, such that both of  the Hardy type inequalities
\begin{equation}\label{eq:HDM'}
 \int_{\mathbb{R_+}}     \int_{t}^{\infty} g(s)ds  {\textstyle\frac{\phi_1(t)}{t}  dt} 
\leq C_1  \int_{\mathbb{R_+}} g(t)  \phi_2 \left( C_2t \right) dt 
\end{equation}
and
\begin{align}\label{eq:HDN'}
 \int_{\mathbb{R_+}}    \left( \int_{0}^{t} g(s)ds \right)^{q/r}  \textstyle\frac{\phi_1(t^{\frac{1}{r}})}{t^{\frac{q-1}{r}+1}  }  dt 
\leq C_1  \int_{\mathbb{R_+}} g(t)^{q/r}   \phi_2 \left( C_2t^{\frac{1}{r}} \right) t^{\frac{1}{r}-1}dt,
\end{align}
hold for all nonnegative, nonincreasing functions $g$ on $\mathbb{R_+}$.
\end{enumerate}
\end{theorem}

\begin{proof}
In view of \Cref{Calderon type theorem for Sqr},  the necessary and sufficient conditions on Young functions $\Phi_1, \Phi_2$ such that the norm inequality (\ref{eq: MT2}) holds are the same as those for which the following norm inequality for $S_{q,r}$ holds,
\begin{equation*}
\| S_{q,r}f^* \|_{L_{\Phi_1}(\mathbb{R_+},m)} \leq C \| f^* \|_{L_{\Phi_2}(\mathbb{R_+},m)}.
\end{equation*}
This norm inequality is, in turn, equivalent to the modular inequality,
\begin{equation}\label{eq:moddDSqr}
  \int_{\mathbb{R_+}}\Phi_1(S_{q,r}f^*(t))dt \leq \int_{\mathbb{R_+}}\Phi_2(Cf^*(s))ds.
\end{equation}
In view of \Cref{LDE for generalized Calderon operator},
\begin{flalign*} 
& \int_{\mathbb{R_+}} \Phi_1 \left(  S_{q,r}f^*(t)  \right)dt  \\
 &= \int_{\mathbb{R_+}} \phi_1(t) m_{S_{q,r}f^*}(t)dt  \\
& \leq E \int_{\mathbb{R_+}} \phi_1(t) \left[ {\left( \textstyle\frac{t}{\beta} \right)^{-1}} \int_{t/ \beta}^{\infty} m_f(\lambda)d\lambda   +     {\left( \textstyle\frac{t}{\beta} \right)^{-q}} \left( \int^{t/ \beta}_{0} m_{f}(\lambda)^{r/q} \lambda^{r-1}d\lambda \right)^{q/r} \right]dt \\
& = E \left[ \int_{\mathbb{R_+}} \frac{\phi_1(t)}{t}   \int_{t}^{\infty} {m_f(\textstyle\frac{\lambda}{\beta})}d\lambda dt   +  \int_{\mathbb{R_+}} \frac{\phi_1(t)}{t^q}    \left( \int_{0}^{t} {m_{f}(\textstyle\frac{\lambda}{\beta})^{r/q}} \lambda^{r-1}d\lambda \right)^{q/r} dt \right] \\
& = E \left[ \int_{\mathbb{R_+}} \frac{\phi_1(t)}{t}   \int_{t}^{\infty} {m_{\beta f}(\lambda)}d\lambda dt   +  \int_{\mathbb{R_+}} \frac{\phi_1(t)}{t^q}    \left( \int_{0}^{t} {m_{ \beta f}(\lambda)^{r/q}} \lambda^{r-1}d\lambda \right)^{q/r} dt \right].
\end{flalign*}
Now, by (\ref{eq:HDM'}),
\begin{equation*}
\int_{\mathbb{R_+}} \frac{\phi_1(t)}{t}   \int_{t}^{\infty} {m_{ \beta f^*}(\lambda)}d\lambda dt \leq C_1  \int_{\mathbb{R_+}} m_{\beta f^*}(\lambda)  \phi_2 \left( C_2 \lambda \right) d\lambda. 
\end{equation*}
Again,
\begin{equation*}
 \int_{\mathbb{R_+}}    \frac{\phi_1(t)}{t^q}\left( \int_{0}^{t} m_{\beta f^*}(\lambda)^{r/q}\lambda^{r-1}d\lambda \right)^{q/r}dt  = {\textstyle\frac{1}{r^{1+ \frac{q}{r}}}}   \int_{\mathbb{R_+}} \frac{\phi_1(t^{\frac{1}{r}})}{t^{\frac{q-1}{r}+1}}  \left( \int_{0}^{t} m_{\beta f^*}(\lambda^{\frac{1}{r}})^{r/q}d\lambda \right)^{q/r}dt
\end{equation*}
which, by (\ref{eq:HDN'}), is dominated by
\begin{equation*}
{\textstyle\frac{C_1}{r^{1+ \frac{q}{r}}}} \int_{\mathbb{R_+}} m_{\beta f^*}(\lambda^{\frac{1}{r}}) \phi_2 \left( C_2 \lambda^{\frac{1}{r}} \right) \lambda^{\frac{1}{r}-1} d\lambda = {\textstyle\frac{C_1}{r^{ \frac{q}{r}}C_2}} \int_{\mathbb{R_+}} m_{C_2\beta f^*}(\lambda) \phi_2 \left(  \lambda \right) d\lambda .
\end{equation*}

So, (\ref{eq:moddDSqr}) holds with $C= \max \left\lbrace 1, \textstyle\frac{2EC_1}{C_2} \right\rbrace C_2 \beta$.

An argument similar to the one above yields, on making the change of variable $t \rightarrow t^r$, then $s \rightarrow s^r$,
\begin{align*}
 \int_{\mathbb{R_+}}    \left( \int_{0}^{t} g(s)ds \right)^{q/r}  \textstyle\frac{\phi_1(t^{\frac{1}{r}})}{t^{\frac{q-1}{r}+1}  }  dt 
  & = {\textstyle r^{1+\frac{q}{r}}} \int_{\mathbb{R_+}}    \left( \int_{0}^{t} g(s^{r}) s^{r-1}ds \right)^{q/r}  \textstyle\frac{\phi_1(t)}{t^{q}  }  dt\\
   & = {\textstyle r^{1+\frac{q}{r}}} \int_{\mathbb{R_+}}    \left( \int_{0}^{t} m_{ f^*}(s)^{r/q}s^{r-1}ds  \right)^{q/r}  \textstyle\frac{\phi_1(t)}{t^{q}  }  dt\\
& \leq 2^q {\textstyle r^{1+\frac{q}{r}}  \left( 1+ \left( \textstyle\frac{r}{q} \right)^{\frac{q}{r}} \right)}  \int_{\mathbb{R_+}} m_{S_{q,r}f^*}(t) \phi_1(t)dt  \\
& = 2^q {\textstyle r^{1+\frac{q}{r}}  \left( 1+ \left( \textstyle\frac{r}{q} \right)^{\frac{q}{r}} \right)} \int_{\mathbb{R_+}} \Phi_1(S_{q,r}f^*(t))dt\\
& \leq 2^q {\textstyle r^{1+\frac{q}{r}}  \left( 1+ \left( \textstyle\frac{r}{q} \right)^{\frac{q}{r}} \right)} \int_{\mathbb{R_+}} \Phi_2(Cf^*(t))dt\\
& = 2^q {\textstyle r^{1+\frac{q}{r}}  \left( 1+ \left( \textstyle\frac{r}{q} \right)^{\frac{q}{r}} \right)} \int_{\mathbb{R_+}} m_{f^*}( {\textstyle\frac{\lambda}{C}}) \phi_2( \lambda)d\lambda  \\
& = 2^q {\textstyle r^{1+\frac{q}{r}}  \left( 1+ \left( \textstyle\frac{r}{q} \right)^{\frac{q}{r}} \right)}C \int_{\mathbb{R_+}} m_{f^*}( s) \phi_2( Cs)ds  \\
& = 2^q {\textstyle r^{1+\frac{q}{r}}  \left( 1+ \left( \textstyle\frac{r}{q} \right)^{\frac{q}{r}} \right)}C \int_{\mathbb{R_+}} g(s^r)^{\frac{q}{r}} \phi_2( Cs)ds  \\
& = 2^q {\textstyle r^{\frac{q}{r}}  \left( 1+ \left( \textstyle\frac{r}{q} \right)^{\frac{q}{r}} \right)}C  \int_{\mathbb{R_+}} g(t)^{q/r}   \phi_2 \left( Ct^{\frac{1}{r}} \right) t^{\frac{1}{r}-1}dt,
\end{align*}
in which we have taken $f^*(t)= \left( g(s^r)^{\frac{q}{r}} \right)^{-1}(t)$ and has made use of (\ref{lower estimate Calderon type operator}) to get (\ref{eq:HDN'}), with $C_1=2^q {\textstyle r^{\frac{q}{r}}  \left( 1+ \left( \textstyle\frac{r}{q} \right)^{\frac{q}{r}} \right)}C $ and $C_2=C$.
\end{proof}

\subsection{\texorpdfstring {The necessity of the Zygmund-Str\"omberg condition for the boundedness of $S_{q,r}$}{}}

\begin{theorem}\label{Necessary condition for Sqr operator}
Fix the indices $q$ and $r$, $1<q< \infty$, $1 \leq r < \infty$. Let $\Phi_1$ and $\Phi_2$ be Young functions such that 
\begin{equation}\label{eq:norm inequality for Sqr}
\parallel S_{q,r}f \parallel _{L_{\Phi_1}(\mathbb{R_+}, m)} \leq C \parallel f \parallel_{L_{\Phi_2}(\mathbb{R_+}, m)},
\end{equation}
in which $C>0$ is independent of  $f \in L_{\Phi_2}(\mathbb{R_+},m)$. Then, there holds the Zygmund-Str\"omberg condition
\begin{equation}\label{eq:Sprcond}
t^q \int_t^{\infty} \frac{\Phi_1(s)}{s^{q+1}}ds \leq A \Phi_2(Bt),
\end{equation}
where the constants $A,B>0$ does not depend on $t \in \mathbb{R_+}$.
\end{theorem}

\begin{proof} According to \Cref{EquvlcHqrnormandmod}, the norm inequality (\ref{eq:norm inequality for Sqr}) holds if  and only if one has the modular inequality
\begin{equation}\label{eq:modddS1}
  \int_{\mathbb{R_+}}\Phi_1(S_{q,r}f^*(t))dt \leq \int_{\mathbb{R_+}}\Phi_2(Kf^*(s))ds ,
\end{equation}
for all $f \in M_{+}(\mathbb{R_+},m)$.

Fix $t \in \mathbb{R_+}$. We will obtain (\ref{eq:Sprcond}) from (\ref{eq:modddS1}) by substituting the function $f(s)= f^*(s)= t \chi_{(0,1)}(s)$ in the modular inequality. Indeed,
\begin{equation*}
\int_{\mathbb{R_+}}\Phi_2(Kf^*(s))ds=\int_{0}^{1}\Phi_2(Kt)ds = \Phi_2(Kt).
\end{equation*}
 Again, for $y < 1 $,
\begin{equation*}
\begin{split}
(S_{q,r}f^*)(y) & = \frac{t}{y} \int_0^y \chi_{(0,1)}(s)ds + t \left( y^{-r/q} \int_y^{1} s^{r/q-1}ds \right)^{\frac{1}{r}}\\
& = t + t \left( \textstyle{\frac{q}{r}} y^{-r/q} \left( 1- y^{r/q} \right) \right)^{\frac{1}{r}}\\
& = t + t (\textstyle{\frac{q}{r}})^{\frac{1}{r}}   \left( (1/y)^{r/q}- 1 \right)^{\frac{1}{r}} \\
& \geq c t \left[  1 + \left( (1/y)^{r/q}- 1 \right)^{\frac{1}{r}}  \right] \\
& = c t (1/y)^{1/q},  
\end{split}
\end{equation*}
where $c= \min\{1,(\textstyle{\frac{q}{r}})^{\frac{1}{r}} \}$. So,
\begin{equation*}
\begin{split}
\int_{\mathbb{R_+}} \Phi_1(S_{q,r}(f^*)(y))dy & \geq \int_0^1 \Phi_1(S_{q,r}(f^*)(y))dy\\
 & \geq \int_0^1 \Phi_1(ct y^{-1/q})dy\\
& = q(ct)^q \int_{c t}^{\infty} \frac{\Phi_1(z)}{z^{q+1}}dz,
\end{split}
\end{equation*}
where  we have made the change of variable $ct y^{-1/q} = z$. Altogether, then, we have
\begin{equation*}
q(ct)^q \int_{c t}^{\infty} \frac{\Phi_1(z)}{z^{q+1}}dz \leq \Phi_2(Kt).
\end{equation*}
Replacing $ct$ by $t$ yields (\ref{eq:Sprcond}), with $A= \frac{1}{q}$ and $B= \frac{K}{c}$.
\end{proof}

\subsection{\texorpdfstring {The case $ 1 \leq r < q$}{}}
In this section, we prove our interpolation result for the class $W((1,1),(q,r);\mu,\nu)$ in the case of $ 1 \leq r < q$, by characterizing the weighted Hardy inequalities obtained in \Cref{Ineq equi Mod Hqr}, using a result of Sawyer \cite[Theorem 2]{Sw90}, which we now state.
\begin{theorem}[E. T. Sawyer, {\cite{Sw90}}]\label{Sawyer wtd Maximal function}
Suppose that $w_1(x)$ and $v_1(x)$ are nonnegative measurable functions on $\mathbb{R_+}$. If $1<p_1 \leq q_1 < \infty$,
then
\begin{equation}\label{eq:Swth2 our form0}
\left( \int_0^{\infty} \left( x^{-1}\int_0^x f(t)dt \right)^{q_1} w_1(x) dx \right)^{\frac{1}{q_1}} \leq C \left( \int_0^{\infty} f(x)^{p_1} v_1(x) dx \right)^{\frac{1}{p_1}},
\end{equation}
holds for all  nonnegative and nonincreasing functions $f$, if and only if both of the following conditions hold:
\begin{equation}
\left( \int_0^t w_1(x)dx \right)^{\frac{1}{q_1}} \leq A \left( \int_0^t v_1(x)dx \right)^{\frac{1}{p_1}},\ \ \  \text{for all} \ t>0;
\end{equation}

\begin{equation}
\left( \int_t^{\infty} x^{-q_1}w_1(x)dx \right)^{\frac{1}{q_1}} \left( \int_0^t \left( x^{-1}V_1(x) \right)^{-p_1'} v_1(x)dx \right)^{\frac{1}{p_1'}} \leq B,\ \ \  \text{for all} \ t>0,
\end{equation}
where $V_1(x)= \int_0^t v_1(y)dy$. Moreover, if $C$ is the best constant in (\ref{eq:Swth2 our form0}) , then $C \approx A+B$.
\end{theorem}

Next we prove our interpolation result.

\begin{theorem}\label{main result for the Wqr class 1}
Let $(X,\mu)$ and $(Y,\nu)$ be $\sigma$-finite measure spaces with $\mu(X)=\nu(Y)=\infty$, the latter being nonatomic and separable. Fix indices $p$ and $r$, where $1 < q<\infty$ and $ 1 \leq r < q$. Suppose $\Phi_{i}(t)= \int_0^t \phi_{i}(s)ds$, $i=1,2$, are Young functions satisfying Zygmund-Str\"omberg condition: There exist $A>0$ such that for all $t \in \mathbb{R_+}$,
\begin{equation}\label{Zymund Stromberg for P in Sqr}
t \int_0^t \frac{\Phi_1(s)}{s^2}ds \leq \Phi_2(At).
\end{equation}
Then, setting $q_1=q/r$, the following are equivalent:  

\begin{enumerate}
\item
 To each $T \in W((1,1),(q,r);\mu,\nu)$ there corresponds $C>0$ such that
\begin{equation}\label{eq: MT3}
 \| Tf \|_{L_{\Phi_1}(Y, \nu)} \leq C \| f \|_{L_{\Phi_2}(X,\mu)},
\end{equation}
for all $f\in L_{\Phi_2}(X,\mu)$;
\item
There exist $C_2>0$ such that the following condition, below, holds 
\begin{align}\label{eq:cianchitypeq}
 \left(\int_{0}^{t}  \frac{\phi_2(C_2y)}{\Phi_2(C_2y)^{q_1'}} y^{rq_1'} dy \right)^{\frac{1}{q_1'}} \left( \int_{t}^{\infty}\frac{ \phi_1(y) }{y^{q}}dy \right)^{\frac{1}{q_1}} \leq F < \infty,
\end{align}
namely,
\begin{align*}
 \left(\int_{0}^{t}  \frac{\phi_2(C_2y)}{\Phi_2(C_2y)^{   \frac{q}{q-r} }} y^{\frac{rq}{q-r}} dy \right)^{1-\frac{r}{q}} \left( \int_{t}^{\infty}\frac{ \phi_1(y) }{y^{q}}dy \right)^{\frac{r}{q}} \leq F < \infty.
\end{align*}
\end{enumerate}
Moreover, if $C$ is the least constant for which \eqref{eq: MT3} holds, then the ratio $C/(A+B)$ is bounded between two positive constants depending only on $q,r$, and $c,C,M_{1}$, $M_{q,r}$ appearing in \Cref{distributional estimate T in class W(11qr)}.
\end{theorem}

\begin{proof} 
In view of \Cref{Ineq equi Mod Hqr}, we need necessary and sufficient conditions on the appropriate weights in order that the inequalities
\begin{equation}\label{eq:HDM'I}
 \int_{\mathbb{R_+}}     \int_{t}^{\infty} g(s)ds  {\textstyle\frac{\phi_1(t)}{t}  dt} 
\leq C_1  \int_{\mathbb{R_+}} g(t)  \phi_2 \left( C_2t \right) dt 
\end{equation}
and
\begin{align}\label{eq:HDN'II}
 \int_{\mathbb{R_+}}    \left( \int_{0}^{t} g(s)ds \right)^{q/r}  \textstyle\frac{\phi_1(t^{\frac{1}{r}})}{t^{\frac{q-1}{r}+1}  }  dt 
\leq C_1  \int_{\mathbb{R_+}} g(t)^{q/r}   \phi_2 \left( C_2t^{\frac{1}{r}} \right) t^{\frac{1}{r}-1}dt,
\end{align}
hold with $C_1,C_2>0$ independent of the nonnegative, nonincreasing functions $g$ on $\mathbb{R_+}$.

Interchanging the order of integration in the integral on the left side of (\ref{eq:HDM'I}) leads to the inequality
\begin{equation*}
 \int_{\mathbb{R_+}}      g(s) \int_0^s \frac{\phi_1(t)}{t}  dt \, ds
\leq C_1  \int_{\mathbb{R_+}} g(t)  \phi_2 \left( C_2t \right) dt.
\end{equation*}
The most general nonnegative, nonincreasing $g$ for which this latter inequality holds essentially has the form
\begin{equation*}
g(s)= \int_s^{\infty} h(y)dy, \ \ \ \text{for some} \ h \in M_+(\mathbb{R_+},m),
\end{equation*}
in which case (\ref{eq:HDM'I}) changes to
\begin{equation*}
 \int_{\mathbb{R_+}}   h(y) \int_0^y \left( \int_0^s \frac{\phi_1(t)}{t}  dt \right) \, ds \, dy 
\leq C_1  \int_{\mathbb{R_+}}  h(y) \int_0^y \phi_2 \left( C_2t \right) dt \, dy.
\end{equation*}
One readily shows this is satisfied if and only if one has (\ref{Zymund Stromberg for P in Sqr}).

As for the inequality, (\ref{eq:HDN'II}), Theorem $2$ of \cite{Sw90} shows it holds if and only if for $t \in \mathbb{R_+}$,
\begin{equation}\label{First condition}
 \int_0^t \phi_1(s^{\frac{1}{r}})s^{\frac{1}{r}-1} ds \leq A  \int_0^t \phi_2(C_2s^{\frac{1}{r}}) s^{\frac{1}{r}-1} ds
\end{equation}
and
\begin{equation}\label{Second condition}
\left( \int_t^{\infty} s^{-q_1} \phi_1(s^{\frac{1}{r}})s^{\frac{1}{r}-1} ds \right)^{\frac{1}{q_1}} \left( \int_0^t \left( s^{-1} \int_0^s \phi_2(C_2y^{\frac{1}{r}}) y^{\frac{1}{r}-1} dy \right)^{-q_1'} \phi_2(C_2s^{\frac{1}{r}}) s^{\frac{1}{r}-1} ds \right)^{\frac{1}{q_1'}} \leq B,
\end{equation}
hold, where $q_1= \frac{q}{r}$.

After suitable change of variable, (\ref{First condition}) reads
\begin{equation*}
\Phi_1(t^{\frac{1}{r}}) \leq \frac{A}{C_2} \Phi_2(C_2t^{\frac{1}{r}}),
\end{equation*}
or, on replacing $t^{\frac{1}{r}}$ by $t$,
\begin{equation*}
\Phi_1(t) \leq  \frac{A}{C_2} \Phi_2(C_2t).
\end{equation*}
But, this condition is implied by the Zygmund-Str\"omberg condition (\ref{Zymund Stromberg for P in Sqr}), which is one of our hypothesis.

The change of variable $s \rightarrow s^r$ in the left hand integral in (\ref{Second condition}) yields
\begin{equation*}
r \int_{t^{\frac{1}{r}}}^{\infty} s^{-q} \phi_1(s)ds
\end{equation*}
Again, as observed above, 
\begin{equation*}
\int_0^s \phi_2(C_2y^{\frac{1}{r}}) y^{\frac{1}{r}-1} dy = \frac{r}{C_2} \Phi_2(C_2s^{\frac{1}{r}}),
\end{equation*}
so that
\begin{align*}
\int_0^t \left( s^{-1} \int_0^s \phi_2(C_2y^{\frac{1}{r}}) \right. & \left.  y^{\frac{1}{r}-1} dy \right)^{-q_1'}  \phi_2(C_2s^{\frac{1}{r}}) s^{\frac{1}{r}-1} ds \\
& = \int_0^t \left(  \frac{rs^{-1}}{C_2} \Phi_2(C_2s^{\frac{1}{r}}) \right)^{-q_1'} \phi_2(C_2s^{\frac{1}{r}}) s^{\frac{1}{r}-1} ds \\
& =r \left( \frac{C_2}{r} \right)^{q_1'} \int_0^{t^{\frac{1}{r}}} \frac{\phi_2(C_2s)}{\Phi_2(C_2s)^{q_1'}}  s^{rq_1'} ds \\*
\end{align*}
Thus, (\ref{Second condition}) amounts to (\ref{eq:cianchitypeq}).
\end{proof}

\subsection{\texorpdfstring {The case $r \geq q$}{}}
Our result in this case is independent of $r$. It is given in
\begin{theorem}\label{main result for the Wqr class 2}
Let $(X,\mu)$ and $(Y,\nu)$ be $\sigma$-finite measure spaces with $\mu(X)=\nu(Y)=\infty$, the latter being nonatomic and separable. Fix the indices  $q$ and $r$,  $r \geq q>1$. Suppose $\Phi_{i}(t)= \int_0^t \phi_{i}(s)ds$, $i=1,2$, are Young functions satisfying Zygmund-Str\"omberg condition (\ref{Zymund Stromberg for P in Sqr}).  Then, the following are equivalent:  
\begin{enumerate}
\item
 To each $T \in W((1,1),(q,r);,\mu,\nu)$ there corresponds $C>0$ such that
\begin{equation}\label{eq: norm inequality for T in Wqr}
 \| Tf \|_{L_{\Phi_1}(Y, \nu)} \leq C \| f \|_{L_{\Phi_2}(X,\mu)},
\end{equation}
for all $f\in L_{\Phi_2}(X,\mu)$;
\item
There exist $A,B>0$ such that (\ref{eq:Sprcond}) holds.
\end{enumerate}
\end{theorem}

The result for $r \geq q$ can be reduced to the case of $r=q$. This is a consequence of

\begin{proposition}\label{sufficient condition Hqr r greater than q}
Given $r \geq q$, there exists a constant $K>0$, depending on $r$, such that 
\begin{equation*}
(S_{q,r} f^*) (t) \leq K (S_{q,q} f^*) (2^{\frac{q}{r}-1}t),
\end{equation*}
for all $f \in M(\mathbb{R_+},m)$ and all $t \in \mathbb{R_+}$.
\end{proposition}
\begin{proof}
It suffices to verify the inequality
\begin{equation*}
(H_{q,r} f^*) (t) \leq K (H_{q,q} f^*) (t),
\end{equation*}
for all $f \in M(\mathbb{R_+},m)$ and all $t \in \mathbb{R_+}$.  Recall,
\begin{equation*}
(H_{q,r} f^* )(t)= \left( t^{-\frac{r}{q}} \int_t^{\infty} f^*(s)^r s^{\frac{r}{q}-1}ds \right)^{\frac{1}{r}}    .
\end{equation*}
Letting $s=u^{\frac{q}{r}}$, we get  
\begin{equation*}
(H_{q,r} f^* )(t)= \left( \left( \frac{q}{r} \right) t^{-\frac{r}{q}} \int_{t^{\frac{r}{q}}}^{\infty} f^*(u^{\frac{q}{r}})^r du \right)^{\frac{1}{r}} 
\end{equation*}
or
\begin{equation*}
\left[  (H_{q,r} f^*) (t^{\frac{q}{r}}) \right]^r=  \left( \frac{q}{r} \right) t^{-1}  \int_{t}^{\infty} f^*(u^{\frac{q}{r}})^r du.
\end{equation*}
Now,
\begin{equation*}
\begin{split}
\left[  (H_{q,r} f^*) (t^{\frac{q}{r}}) \right]^r & =  \left( \frac{q}{r} \right) t^{-1}  \int_{\mathbb{R_+}}  f^* \left( (t+u)^{\frac{q}{r}} \right)^r du\\
& \leq \left( \frac{q}{r} \right) t^{-1}  \int_{\mathbb{R_+}}  f^* \left(2^{\frac{q}{r}-1} t^{\frac{q}{r}} + 2^{\frac{q}{r}-1}u^{\frac{q}{r}} \right)^r du  \ \ \ (\text{since} \, r \geq q)   \\
& = \gamma^{-\frac{r}{q}}  t^{-1}  \int_{\mathbb{R_+}}  f^* \left( x + \gamma t^{\frac{q}{r}}  \right)^r x^{\frac{r}{q}-1} dx,
\end{split}
\end{equation*}
where $\gamma=2^{\frac{q}{r}-1}$.

Thus,
\begin{equation*}
\begin{split}
(H_{q,r} f^*) (t^{\frac{q}{r}}) & \leq \gamma^{-\frac{1}{q}} {\left( \textstyle\frac{q}{r} \right)^{\frac{1}{r}}} t^{-\frac{1}{r}} \| f^*(\cdot + \gamma t^{\frac{q}{r}}) \|_{L_{q,r}(\mathbb{R_+},m)}\\
& \leq \gamma^{-\frac{1}{q}} {\left( \textstyle\frac{q}{r} \right)^{\frac{1}{r}}} t^{-\frac{1}{r}} \| f^*(\cdot +  \gamma t^{\frac{q}{r}}) \|_{L_{q,q}(\mathbb{R_+},m)}\\
& = \gamma^{-\frac{1}{q}} {\left( \textstyle\frac{q}{r} \right)^{\frac{1}{r}}} t^{-\frac{1}{r}} \left(  \int_{\gamma t^{\frac{q}{r}}}^{\infty}  f^* \left(  s \right)^q  ds \right)^{\frac{1}{q}}.
\end{split}
\end{equation*}
Replacing $t$ by $t^{r/q}$ yields

\begin{equation*}
(H_{q,r} f^*) (t) \leq  {\left( \textstyle\frac{q}{r} \right)^{\frac{1}{r}}} { ( \gamma t )}^{-\frac{1}{q}} \left(  \int_{\gamma t}^{\infty}  f^* \left(  s \right)^q  ds \right)^{\frac{1}{q}} =  {\left( \textstyle\frac{q}{r} \right)^{\frac{1}{r}}}  (H_{q,q} f^* )(\gamma t).
\end{equation*}
\end{proof}

\begin{proof}[Proof of \Cref{main result for the Wqr class 2}]
In view of Theorems \ref{Calderon type theorem for Sqr} and \ref{EquvlcHqrnormandmod},
the assertion in $1$ is equivalent to the requirement that
\begin{equation}\label{modular inequality for Sqr required for final one for r bigger than q}
 \int_{\mathbb{R_+}}\Phi_1(S_{q,r}f^*(t))dt \leq \int_{\mathbb{R_+}}\Phi_2(Kf^*(s))ds,
\end{equation}
for all $f \in M(\mathbb{R_+},m)$.

Suppose first that (\ref{modular inequality for Sqr required for final one for r bigger than q}) holds. This, together with \Cref{sufficient condition Hqr r greater than q}, shows $2$ holds, given $1$.

Assume, next, we have $2$. According to \Cref{LDE for generalized Calderon operator},
\begin{align*}
\int_{\mathbb{R_+}} \Phi_1 &(S_{q,q}f^*(t))dt  = \int_{\mathbb{R_+}} \phi_1(t)m_{S_{q,q}f^*}(t)dt\\*
& \leq E \left[ \int_{\mathbb{R_+}} \phi_1(t) \left( \frac{t}{ \beta} \right)^{-1}      \int_{\frac{t}{\beta}}^{\infty} m_{f^*}(\lambda) d \lambda \,  dt   + \int_{\mathbb{R_+}}   \frac{\phi_1(t)}{(t/ \beta)^q}    \int_0^{\frac{t}{\beta}} m_{f^*}(\lambda) \lambda^{q-1} d \lambda \, dt    \right] .
\end{align*}
To begin with the first term in the last expression, we have
\begin{align*}
\int_{\mathbb{R_+}} \phi_1(t) \left( \frac{t}{\beta} \right)^{-1}      \int_{\frac{t}{\beta}}^{\infty} m_{f^*}(\lambda) d \lambda \,  dt 
& = \beta \int_{\mathbb{R_+}} m_{f^*}(\lambda) \int_0^{\beta \lambda} \frac{\phi_1(t)}{t}dt \, d \lambda \\
& \leq \beta \int_{\mathbb{R_+}} m_{f^*}(\lambda) \int_0^{\beta \lambda} \frac{\Phi_1(2t)}{t^2}dt \, d \lambda \\
& =  \int_{\mathbb{R_+}} m_{f^*}(\lambda / 2 \beta) \int_0^{ \lambda} \frac{\Phi_1(t)}{t^2}dt \, d \lambda \\
& \leq A \int_{\mathbb{R_+}} m_{2 \beta f^*}(\lambda) \frac{\Phi_2(B \lambda)}{ \lambda} \, d \lambda \\
& \leq A \int_{\mathbb{R_+}} \Phi_2( 2 \beta Bf^*(t)) \, d t.
\end{align*}
Finally,
\begin{align*}
\int_{\mathbb{R_+}}   \frac{\phi_1(t)}{(t/ \beta)^q}    \int_0^{\frac{t}{\beta}} m_{f^*}(\lambda) \lambda^{q-1} d \lambda \, dt  
& = \beta^q \int_{\mathbb{R_+}} m_{f^*}(\lambda) \lambda^{q-1} \int_{\beta \lambda}^{\infty} \frac{\phi_1(t)}{t^q}dt \, d \lambda \\
& \leq \beta^q \int_{\mathbb{R_+}} m_{f^*}(\lambda) \lambda^{q-1} \int_{\beta \lambda}^{\infty} \frac{\Phi_1(2t)}{t^{q+1}}dt \, d \lambda \\
& =  \int_{\mathbb{R_+}} m_{f^*}(\lambda / 2\beta) \lambda^{q-1} \int_{ \lambda}^{\infty} \frac{\Phi_1(t)}{t^{q+1}}dt \, d \lambda \\
& \leq A \int_{\mathbb{R_+}} m_{2 \beta f^*}(\lambda) \frac{\Phi_2(B \lambda)}{ \lambda} \, d \lambda \\
& \leq A \int_{\mathbb{R_+}} \Phi_2(2 \beta Bf^*(t)) \, d t.
\end{align*}

\end{proof}

\section{Interpolation pairs of Orlicz spaces for the class $W((p_0,r_0),(p_1 ,r_1);\mu,\nu)$}\label{Class W(p0r0,p1r1)}

%


\subsection{A Calder\'on-type theorem}


\begin{theorem}\label{Dominance of Calderon operator for joined class}
Let $(X,\mu)$ and $(Y,\nu)$ be $\sigma$-finite measure spaces. Fix the indices $p_0,p_1, r_0$ and $r_1$, $1 < p_0 < p_1 < \infty$ and $ 1 \leq r_1 , r_2 < \infty $. Then, given $T \in W((p_0,r_0),(p_1 ,r_1);\mu,\nu)$, one has
\begin{equation}\label{6eq:Spr1qr2}
(Tf)^*(t)\leq K [ (H^{p_0,r_0} + H_{p_1,r_1}) f^* ](ct),
\end{equation}
in which $K=K(T)>0$ and $c=c(T)>0$ are independent of  $f \in (L_{p_0,r_0}+L_{p_1,r_1})(X,\mu)$ and of $t \in \mathbb{R_+}$.
\end{theorem}
 \begin{proof}
 Let $ f\in(L_{p_0,r_0}+L_{p_1,r_1})(X,\mu)$ and fix $t \in \mathbb{R_+}$.
At $x \in X$, set
\begin{equation*}
   f_1(x)= \min[|f(x)|, f^*(t)] \, \text{sgn}f(x)
\end{equation*}
and 
\begin{equation*}
f_0(x)=f(x)-f_1(x)=[|f(x)|-f^*(t)]^+ \, \text{sgn}f(x).
\end{equation*}
Then, $f= f_0 + f_1$ and for all $s \in \mathbb{R_+}$
\begin{equation*}
f^*_0(s)= [f^*(s)-f^*(t)]^+,
\end{equation*}
\begin{equation*}
f_1^*(s)= \min (f^*(s),f^*(t)).
\end{equation*}

Moreover, as shown in \Cref{membership of splitted parts}, $f_0 \in L_{p_0,r_0}(X,\mu)$ and $f_1 \in L_{p_1,r_1}(X,\mu)$.

So, if $T$ has $r$-quasilinearity constants $C$ and $c$ (see (\ref{nu quasilinear}), p. \pageref{nu quasilinear}),
\begin{equation*}
\begin{split}
(Tf)^*(t) & \leq C [(Tf_0)^*(ct)+(Tf_1)^*(ct)]\\
& \leq C \left[ (ct)^{-\frac{1}{p_0}}M_{p_0,r_0} \| f_0 \|_{L_{p_0,r_0}(X,\mu)} + (ct)^{-\frac{1}{p_1}}M_{p_1,r_1} \| f_1 \|_{L_{p_1,r_1}(X,\mu)} \right],
\end{split}
\end{equation*}
in which $M_{p_i,r_i}$ is the norm of $T$ as a mapping from $L_{p_i,r_i}(X,\mu)$ to $L_{p_i,\infty}(Y,\nu)$, $i=0,1$.

Now,
\begin{equation}\label{6eq:lpr}
\begin{split}
\parallel f_0 \parallel_{L_{p_0,r_0}(X,\mu)}
 & = \parallel f_0^* \parallel_{L_{p_0,r_0}(\mathbb{R_+},m)}\\
& = {\left( \textstyle\frac{r_0}{p_0} \right)^{\frac{1}{r_0}}} \left( \int_0^t (f^*(s)-f^*(t))^{r_0} s^{\frac{r_0}{p_0}-1} ds \right)^{\frac{1}{r_0}} \\
& \leq {\left( \textstyle\frac{r_0}{p_0} \right)^{\frac{1}{r_0}}} \left( \int_0^t f^*(s)^{r_0} s^{\frac{r_0}{p_0}-1} ds \right)^{\frac{1}{r_0}} \\
& \leq {\left( \textstyle\frac{r_0}{p_0} \right)^{\frac{1}{r_0}}}   c^{-\frac{1}{r_0}} \left( \int_0^{ct} f^*(s)^{r_0} s^{\frac{r_0}{p_0}-1} ds \right)^{\frac{1}{r_0}}
\end{split}
\end{equation}
and
\begin{align*}
\parallel f_1 \parallel_{L_{p_1,r_1}(X,\mu)}
 & = \parallel f_1^* \parallel_{L_{p_1,r_1}(\mathbb{R_+},m)}\\*
& = {\left( \textstyle\frac{r_1}{p_1} \right)^{\frac{1}{r_1}}}  \left( \int_0^t f^*(t)^{r_1} s^{\frac{r_1}{p_1}-1} ds +\int_t^{\infty} f^*(s)^{r_1} s^{\frac{r_1}{p_1}-1} ds \right)^{\frac{1}{r_1}} \\*
& = {\left( \textstyle\frac{r_1}{p_1} \right)^{\frac{1}{r_1}}}   \left( {\textstyle\frac{p_1}{r_1}} f^*(t)^{r_1} t^{\frac{r_1}{p_1}}  +\int_t^{\infty} f^*(s)^{r_1} s^{\frac{r_1}{p_1}-1} ds \right)^{\frac{1}{r_1}} \\*
& \leq  {\left( \textstyle\frac{r_1}{p_1} \right)^{\frac{1}{r_1}}}   \left( \left( 1-c^{\frac{r_1}{p_1}} \right)^{-1} \int_{ct}^{t} f^*(s)^{r_1} s^{\frac{r_1}{p_1}-1}ds  +\int_t^{\infty} f^*(s)^{r_1} s^{\frac{r_1}{p_1}-1} ds \right)^{\frac{1}{r_1}} \\*
& \leq  {\left( \textstyle\frac{r_1}{p_1} \right)^{\frac{1}{r_1}}}    \left( 1-c^{\frac{r_1}{p_1}} \right)^{-\frac{1}{r_1}} \left( \int_{ct}^{\infty} f^*(s)^{r_1} s^{\frac{r_1}{p_1}-1} ds \right)^{\frac{1}{r_1}}.
\end{align*}
Altogether, then, 
\begin{equation*}
(Tf)^*(t)  \leq K \left[ \left( H^{p_0,r_0} + H_{p_1,r_1} \right)f^*\right](ct),
\end{equation*}
with $K=  \left( {\left( \textstyle\frac{r_0}{p_0} \right)^{\frac{1}{r_0}}}   c^{-\frac{1}{r_0}}+ {\left( \textstyle\frac{r_1}{p_1} \right)^{\frac{1}{r_1}}}    \left( 1-c^{\frac{r_1}{p_1}} \right)^{-\frac{1}{r_1}} \right) \left( M_{p_0,r_0} + M_{p_1,r_1} \right)C$.
\end{proof}

\begin{theorem}\label{Calderon type theorem for the joined class}
Fix the indices $p_0,p_1,r_0$ and $r_1$, $1 < p_0 <p_1 < \infty$, $1 \leq r_0,r_1 <\infty$.  Let  $(X,\mu)$ and $(Y,\nu)$ be $\sigma$-finite  measure spaces, with $\mu(X)=\nu(Y)=\infty$ and $(Y,\nu)$ being  nonatomic and separable. Assume $\Phi_{i}(t)= \int_0^t \phi_{i}(s)ds$, $i=1,2$, are Young functions.  Then, the  following are equivalent:
\begin{enumerate}
\item[(1)] Every operator $T \in  W((p_0,r_0),(p_1,r_1);\mu,\nu)$ maps $L_{\Phi_2}(X,\mu)$ boundedly into $L_{\Phi_1}(Y,\nu)$;
\item[(2)] The operator $H^{p_0,r_0} + H_{p_1,r_1}$ maps the nonincreasing functions in   $L_{\Phi_2}(\mathbb{R_+},m)$ boundedly into $L_{\Phi_1}(\mathbb{R_+},m)$.
\end{enumerate}
\end{theorem}
\begin{proof}
Suppose $(2)$ holds. Then, given $T \in W((p_0,r_0),(p_1,r_1);\mu,\nu)$ and $f \in L_{\Phi_2}(X,\mu)$, one has, by \Cref{Dominance of Calderon operator for joined class},
\begin{align*}
\| Tf\|_{L_{\Phi_1}(Y,\nu)} & = \| (Tf)^*(t) \|_{L_{\Phi_1}(\mathbb{R_+},m)} \\
& \leq K \| \left[ \left( H^{p_0,r_0} + H_{p_1,r_1} \right) f^* \right](ct)  \|_{L_{\Phi_1}(\mathbb{R_+},m)} \\
& \leq K h(c) \| \left[ \left( H^{p_0,r_0} + H_{p_1,r_1} \right) f^* \right](t)  \|_{L_{\Phi_1}(\mathbb{R_+},m)} \\
& \leq K h(c) C \| f^* \|_{L_{\Phi_2}(\mathbb{R_+},m)}\\
& = K h(c) C \| f\|_{L_{\Phi_2}(X,\mu)},
\end{align*}
namely, $(1)$ holds. 


The argument that $(1)$ implies $(2)$ is by now a familiar one. First, one readily proves that
\begin{equation*}
H^{p_0,r_0} + H_{p_1,r_1} \in W((p_0,r_0),(p_1,r_1);m,m).
\end{equation*}


Thus, second, the operator $  \left( H^{p_0,r_0} + H_{p_1,r_1} \right)^{\sim} $, constructed in \Cref{construction of Pull back of operator S}, is in $ W((p_0,$ $r_0),(p_1, r_1);\mu,\nu)$ and so maps $L_{\Phi_2}(X,\mu)$ boundedly into $L_{\Phi_1}(Y,\nu)$. Third, taking in \Cref{construction of Pull back of operator S}, $X= \mathbb{R_+}$, $\mu=m$, $(Y,\nu)$ to be $(X,\mu)$ and, as $S$, the operator $f \rightarrow f^{*_m}$, one gets, for $f \in M_+(\mathbb{R_+},m)$, a function $\tilde{f} \in M_+(X,\mu)$ such that for all $t \in \mathbb{R_+}$
\begin{equation*}
\tilde{f}^{*_{\mu}}(t) = f^{*_{m}}(t).
\end{equation*}
Therefore, since
\begin{equation*}
\left[  \left( H^{p_0,r_0} + H_{p_1,r_1} \right)^{\sim}f \right]^{*_{\nu}}(t) = \left[  \left( H^{p_0,r_0} + H_{p_1,r_1} \right)f^{*_{\mu}} \right]^{*_{m}} (t)=\left[  \left( H^{p_0,r_0} + H_{p_1,r_1} \right)f^{*_{\mu}} \right] (t),
\end{equation*}
$f \in M_+(\mathbb{R_+},m), t \in \mathbb{R_+}$, we get, for $g \in M_+(\mathbb{R_+},m)$, $g$ nonincreasing, 
\begin{align*}
 \|  \left( H^{p_0,r_0} + H_{p_1,r_1} \right) g \|_{L_{\Phi_1}(\mathbb{R_+},m)}
 & =  \|  \left( H^{p_0,r_0} + H_{p_1,r_1} \right) g^{*_{m}} \|_{L_{\Phi_1}(\mathbb{R_+},m)} \\
& = \|  \left( H^{p_0,r_0} + H_{p_1,r_1} \right) \tilde{g}^{*_{\mu}} \|_{L_{\Phi_1}(\mathbb{R_+},m)} \\
& = \|  \left[  \left( H^{p_0,r_0} + H_{p_1,r_1} \right) \tilde{g}^{*_{\mu}} \right]^{*_{m}} \|_{L_{\Phi_1}(\mathbb{R_+},m)} \\
 & =  \| [ \left( H^{p_0,r_0} + H_{p_1,r_1} \right)^{\sim} {\tilde{g}} ]^{*_{\nu}} \|_{L_{\Phi_1}(\mathbb{R_+},m)} \\
 & =  \|  \left( H^{p_0,r_0} + H_{p_1,r_1} \right)^{\sim} {\tilde{g}}  \|_{L_{\Phi_1}(Y,\nu)} \\
&  \leq C \| \tilde{g}\|_{L_{\Phi_2}(X,\mu)}\\
 &  = C \| {\tilde{g}}^{*_{\mu}} \|_{L_{\Phi_2}(\mathbb{R_+},m)}\\
 &     = C \| {g}^{*_{m}} \|_{L_{\Phi_2}(\mathbb{R_+},m)}\\
 &     = C \| g \|_{L_{\Phi_2}(\mathbb{R_+},m)},
\end{align*}
given $(1)$.
\end{proof}
\subsection{Proof of the main Theorem}
We are now able to verify the main result of this thesis, namely, to give the
\begin{proof}[Proof of \cref{Combined main theorem}]
\Cref{Calderon type theorem for the joined class} ensures that $(1)$ amounts to the assertion that
\begin{equation*}
\|  \left( H^{p_0,r_0} + H_{p_1,r_1} \right) f^* \|_{L_{\Phi_1}(\mathbb{R_+},m)} \leq C \| f^* \|_{L_{\Phi_2}(\mathbb{R_+},m)}
\end{equation*}
with $C>0$ independent of $f \in L_{\Phi_2}(\mathbb{R_+},m)$.

Since $H^{p_0,r_0} + H_{p_1,r_1}$ commutes with dilations, Theorem A in \cite{KRS17} guarantees this assertion equivalent to the inequality
\begin{equation}\label{modular joint calderon operator}
\int_{\mathbb{R_+}} \Phi_1 \left( \left[ \left( H^{p_0,r_0} + H_{p_1,r_1} \right)f^* \right] (t) \right) \leq  \int_{\mathbb{R_+}} \Phi_1 \left( K f^*(s) \right)ds,
\end{equation}
in which $K>0$ is independent of $f \in M_+(\mathbb{R_+},m)$. Indeed, the methods of \Cref{Calderon type theorem for the joined class} shows (\ref{modular joint calderon operator}) equivalent to $(2)$.

Finally, it follows from Theorems \ref{main result 1 class Wpr modified by Kerman}, \ref{main result 22 class Wpr} and \ref{main result for the Wqr class 1}, \ref{main result for the Wqr class 2} that (\ref{modular joint calderon operator}) is equivalent to $(3)$.

\end{proof}

\section{On the monotonicity in $r$ of the condition for $H^{p,r}: L_{\Phi_2} \rightarrow L_{\Phi_1}$}\label{Comparison of the conditions}
We will show in this section that the necessary and sufficient condition for
\begin{equation}\label{Hpr operator for comparision}
H^{p,r}: L_{\Phi_2}(\mathbb{R_+},m) \rightarrow L_{\Phi_1}(\mathbb{R_+},m), \ \ \  1 \leq r<p, 
\end{equation}
namely, with $p_1 = {\textstyle\frac{p}{r}}$
\begin{equation}\label{Cianchi-type condition for comparsision wrt r}
\left( \int_t^{\infty} \frac{\phi_2(Ds)}{\Phi_2(Ds)^{p_1'}}s^{r p_1'} ds \right)^{\frac{1}{p_1'}}  \left( \int_0^{t}\frac{ \phi_1(s) }{s^{p}}ds \right)^{\frac{1}{p_1}} \leq B,  \ \ \text{for all} \ t \in \mathbb{R_+},
\end{equation}
decreases \emph{strictly} in strength as $r$ increases in $[1,p)$.

The same can be shown about the condition for
\begin{equation*}
H_{q,r}: L_{\Phi_2}(\mathbb{R_+},m) \rightarrow L_{\Phi_1}(\mathbb{R_+},m).
\end{equation*}
Now, the inequality
\begin{equation*}
\left( H^{p,r_2} f^* \right)(t) \leq C \left( H^{p,r_1} f^* \right)(t), \ \ \ \text{for all} \ t \in \mathbb{R_+}, \ \ \ 1 \leq r_1 < r_2 < \infty,
\end{equation*}
which follows from (\ref{Lorentz mono}), implies the condition (\ref{Cianchi-type condition for comparsision wrt r}) must decrease in strength. That the decrease is strict in $[1,p)$ is demonstrated by

\begin{example}
Let $1<p< \infty $ and $1 \leq r_1 < r_2 <p$. Let us denote $p_1=p / r_1$ and $p_2= p/r_2$. Fix indices $\alpha_1, \alpha_2$ and $\beta$ with $0 < 1+ \alpha_1  \leq p_1 - p_2$, $p_2 < 1 + \alpha_2 \leq p_1$ and $\beta>p$. Consider the Young functions defined by
\begin{equation}\label{range Young function}
    \Phi_1(t) = \begin{cases}
                 t^{\beta}, & t < e,\\
                 t^{p}(\log t )^{\alpha_1},  & t>e
                \end{cases}
\end{equation}
and
\begin{equation}\label{domain Young function}
    \Phi_2(t) = \begin{cases}
                 t^{\beta}, & t < e,\\
                 t^{p}(\log t )^{\alpha_2},  & t>e.
                \end{cases}
\end{equation}
\end{example}
Here, we will see that for the pair of the Young functions defined  in (\ref{range Young function}) and (\ref{domain Young function}), the condition (\ref{Cianchi-type condition for comparsision wrt r}) holds with $r=r_2$ but not with $r=r_1$.

Observe that the condition (\ref{Cianchi-type condition for comparsision wrt r}) can be rewritten as
\begin{equation}\label{Cianchi type condition Kerman rewritten for comparsision}
\left( \int_{D_i x}^{\infty} \left( \frac{y}{\Phi_2(y^{\frac{1}{r_i}})} \right)^{p'_i-1} dy \right)^{\frac{1}{p'_i}}  \left( \int_0^{x}\frac{ \Phi_1(y^{\frac{1}{r_i}}) }{y^{p_i+1}}dy \right)^{\frac{1}{p_i}} \leq B_i', 
\end{equation}
for some $B_i'>0$, $i=1,2$ and all $x \geq 0$. We can ignore $D_i$ in (\ref{Cianchi type condition Kerman rewritten for comparsision}) for our purpose in this proposition.

First we will see that for these Young functions, the condition (\ref{Cianchi type condition Kerman rewritten for comparsision}) with $i=1$ (that is for $r_1$) does not hold. Let us consider the second integral in the right hand side of (\ref{Cianchi type condition Kerman rewritten for comparsision}). For large $x$ we have,
\begin{equation}\label{Ir1}
\begin{split}
I_{r_1}(x) & := \int_0^{x}\frac{ \Phi_1(y^{\frac{1}{r_1}}) }{y^{p_1+1}}dy \\
& = \int_0^{e^{r_1}} y^{\beta / r_1 -p_1 -1} dy + \int_{e^{r_1}}^{x} \frac{ y^{p/r_1}  (\log (y^{\frac{1}{r_1}}) )^{\alpha_1}    }{y^{p_1+1}}dy \\
& =  \frac{e^{\beta -p }}{\beta / r_1 -p_1} + \frac{1}{r_1^{\alpha_1}} \int_{e^{r_1}}^x (\log y)^{\alpha_1} \frac{dy}{y} \\
& =  \frac{e^{\beta -p }}{\beta / r_1 -p_1} + \frac{1}{r_1^{\alpha_1}} \int_{r_1}^{\log x} \frac{1}{y^{-\alpha_1}} dy \\
& \approx (\log x)^{1 + \alpha_1}.
\end{split}
\end{equation}

For the first integral in the left hand side of (\ref{Cianchi type condition Kerman rewritten for comparsision}), we have
\begin{equation}\label{Jr1}
\begin{split}
J_{r_1}(x) & :=  \int_{ x}^{\infty} \left( \frac{y}{\Phi_2(y^{\frac{1}{r_1}})} \right)^{p'_1-1} dy \\
& = \int_{ x}^{\infty} \left( \frac{y}{   y^{p/r_1} ( \log (y^{1/r_1}))^{\alpha_2}    } \right)^{p'_1-1} dy \\
& = r_1^{\alpha_2(p_1'-1)} \int_{x}^{\infty} \frac{1}{(\log y)^{\alpha_2(p_1'-1)}} \frac{dy}{y} \\
&\approx \int_{\log x}^{\infty} \frac{1}{y^{\alpha_2(p_1'-1)}} dy \  \rightarrow \infty,
\end{split}
\end{equation}
as $\alpha_2(p_1'-1) \leq 1$. So from (\ref{Ir1}) and (\ref{Jr1}) we have that the condition (\ref{Cianchi type condition Kerman rewritten for comparsision}) for $r_1$ does not hold.

It remains to show that the condition (\ref{Cianchi type condition Kerman rewritten for comparsision}) holds for $r_2$. Which follows from the following expressions for the two integrals appearing in (\ref{Cianchi type condition Kerman rewritten for comparsision}) and the assumption $1 - \left( \frac{\alpha_2 - \alpha_1}{p_2} \right) \leq 0$. 
\begin{equation}\label{Ir2}
\begin{split}
I_{r_2}(x) & := \int_0^{x}\frac{ \Phi_1(y^{\frac{1}{r_2}}) }{y^{p_2+1}}dy \\
& = \begin{cases}
                 \frac{x^{\beta / r_2 -p_2 }}{\beta / r_2 -p_2}, & x < e^{r_2},\\
                 \frac{e^{\beta  -p }}{\beta / r_2 -p_2} + \frac{1}{r_2^{\alpha_2}} \frac{(\log x)^{1 + \alpha_2}-r_2^{1+ \alpha_2}}{1+ \alpha_2},  & x>e^{r_2}
                \end{cases}
\end{split}
\end{equation}
and as in (\ref{Jr1})
\begin{equation}\label{Jr2}
\begin{split}
J_{r_2}(x) & :=  \int_{ x}^{\infty} \left( \frac{y}{\Phi_2(y^{\frac{1}{r_2}})} \right)^{p'_2-1} dy \\
& = \begin{cases}
                 \int_x^{e^{r_2}} \left( \frac{y}{y^{\beta / r_2}} \right)^{p_2'-1}dy +  \int_{e^{r_2}}^{\infty} \left( \frac{y}{   y^{p/r_2} ( \log (y^{1/r_2}))^{\alpha_2}    } \right)^{p'_2-1} dy, & x < e^{r_2},\\
                  r_2^{\alpha_2(p_2'-1)} \int_{x}^{\infty} \frac{1}{(\log y)^{\alpha_2(p_2'-1)}} \frac{dy}{y},  & x>e^{r_2}
                \end{cases}\\
& = \begin{cases}
                 \frac{1}{p_2'(\beta /p -1)}  \left[ x^{-p_2'(\beta /p -1)} - e^{-r_2 p_2'(\beta /p -1)}  \right] + A, & x < e^{r_2},\\
                    \frac{   p_2 r_2^{\alpha_2(p_2'-1)}    }{1+ \alpha_2 -p_2} \left( \frac{1}{\log x} \right)^{\frac{p_2' (1+ \alpha_2-p_2)}{p_2}} ,  & x>e^{r_2},
    \end{cases}
\end{split}
\end{equation}
where $A = r_2^{\alpha_2(p_2'-1)} \int_{e^{r_2}}^{\infty}  \frac{1}{    ( \log (y))^{\alpha_2(p_2'-1)}    } \frac{dy}{y} = \frac{  p_2 r_2    }{p_2'(1+ \alpha_2-p_2)}< \infty $ as $\alpha_2(p_2'-1) >1$.

\end{document}